\documentclass[reqno,11pt]{amsart}
\usepackage[utf8]{inputenc}
\usepackage[margin=1in]{geometry}
\usepackage{amscd,amsfonts,amsmath,amssymb,amsthm}
\usepackage{bbm,dsfont,centernot,mathtools}
\usepackage{color,mathrsfs,stackrel}
\usepackage[dvipsnames]{xcolor}
\usepackage[colorlinks=true,linkcolor=RoyalBlue,citecolor=Red]{hyperref}

\newtheorem{corollary}{Corollary}[section]
\newtheorem{lemma}[corollary]{Lemma}
\newtheorem{proposition}[corollary]{Proposition}
\newtheorem{theorem}[corollary]{Theorem}

\theoremstyle{definition}

\newtheorem{remark}[corollary]{Remark}

\numberwithin{equation}{section}

\allowdisplaybreaks

\DeclareMathOperator*{\esssup}{esssup}
\DeclareMathOperator*{\essinf}{essinf}

\def\diam {\mathop {\rm diam}\nolimits}
\def\Lip {\mathop {\rm Lip}\nolimits}
\def\loc {\mathop {\rm loc}\nolimits}
\def\skw {\mathop {\rm skew}\nolimits}
\def\sym {\mathop {\rm sym}\nolimits}
\def\div {\mathop {\rm div}\nolimits}
\def\cof {\mathop {\rm cof}\nolimits}
\def\dist {\mathop {\rm dist}\nolimits}
\def\det {\mathop {\rm det}\nolimits}
\def\tr {\mathop {\rm tr}\nolimits}
\def\de {\mathrm{d}}
\def\R {\mathbb R}
\def\N {\mathbb N}
\def\e {\mathrm{e}}

\title[Geometric rigidity on Sobolev spaces with variable exponent]{Geometric rigidity on Sobolev spaces with variable exponent and applications}

\author[S. Almi]{Stefano Almi}
\address[Stefano Almi]{Department of Mathematics and Applications ``R.~Caccioppoli'', University of Naples Federico II, Via Cintia, Monte S. Angelo, 80126 Napoli, Italy.}
\email{stefano.almi@unina.it}

\author[M. Caponi]{Maicol Caponi}
\address[Maicol Caponi]{Department of Mathematics and Applications ``R.~Caccioppoli'', University of Naples Federico II, Via Cintia, Monte S. Angelo, 80126 Napoli, Italy.}
\email{maicol.caponi@unina.it}

\author[M. Friedrich]{Manuel Friedrich} 
\address[Manuel Friedrich]{Department of Mathematics, Friedrich-Alexander Universit\"at Erlangen-N\"urnberg. Cauerstr.~11,
D-91058 Erlangen, Germany, \& Mathematics M\"{u}nster, 
University of M\"{u}nster, Einsteinstr.~62, D-48149 M\"{u}nster, Germany}
\email{manuel.friedrich@fau.de}

\author[F. Solombrino]{Francesco Solombrino}
\address[Francesco Solombrino]{Department of Mathematics and Applications ``R.~Caccioppoli'', University of Naples Federico II, Via Cintia, Monte S. Angelo, 80126 Napoli, Italy.}
\email{francesco.solombrino@unina.it}

\date{}

\makeatletter
\@namedef{subjclassname@2020}{%
 \textup{2020} Mathematics Subject Classification}
\makeatother

\begin{document}



\maketitle

%---------------
% Introduction
%---------------

\begin{abstract}
We present extensions of rigidity estimates and of Korn's inequality to the setting of (mixed) variable exponents growth. The proof techniques, based on a classical covering argument, rely on the log-H\"older continuity of the exponent to get uniform regularity estimates on each cell of the cover, and on an extension result {\em \`a la} {\sc Nitsche} in Sobolev spaces with variable exponents. As an application, by means of $\Gamma$-convergence we perform a passage from nonlinear to linearized elasticity under variable subquadratic energy growth far from the energy well. 
\end{abstract}

%-----------------------------------------------------
% Lebesgue and Sobolev spaces with variable exponent
%-----------------------------------------------------

\section{Introduction}

{\sc Liouville}'s rigidity result states that smooth mappings are necessarily affine if their gradient is a rotation everywhere. After several qualitative generalizations over the last  decades~\cite{John:1961, kohn-rig, neff.muench, Reshetnyak:1961}, a breakthrough advancement was obtained by {\sc
Friesecke, James, and M\"uller}~\cite{FJM} with a quantitative geometric rigidity estimate in nonlinear elasticity theory. In its basic form, it states that for a sufficiently smooth bounded domain $\Omega \subset \R^n$ and mappings $y \in H^1(\Omega;\R^n)$ there exists a corresponding rotation $R \in SO(n)$ such that
\begin{align}\label{eq: rigidity}
\int_\Omega |\nabla y - R |^2 \, {\rm d}x \le C \int_\Omega\dist^2\big(\nabla y,SO(n)\big) \, {\rm d} x
\end{align}
for a constant $C>0$ only depending on $\Omega$. Subsequently, the result was generalized to general exponents $p \in (1,\infty)$, see~\cite{conti.schweizer}, and to settings of mixed growth~\cite{CDM} stating that for a decomposition
$$\dist\big(\nabla y,SO(n)\big) = f + g \quad \quad \text{a.e.\ for some $f \in L^p(\Omega)$ and $g \in L^q(\Omega)$} $$
for $1 < p < q < \infty$ there exists a corresponding rotation $R \in SO(n)$ and $F \in L^p(\Omega;\R^{n \times n})$, $G \in L^q(\Omega;\R^{n \times n})$ such that 
\begin{align}\label{eq:rigidity2}
\text{$\nabla y - R = F + G$ \ a.e., } \quad \quad \Vert F \Vert_{L^p(\Omega)} \le C\Vert f \Vert_{L^p(\Omega)}, \ \ \Vert G \Vert_{L^q(\Omega)} \le C\Vert g \Vert_{L^q(\Omega)}. 
\end{align}
(For $g =G = 0$ and $p=2$,~\eqref{eq:rigidity2} reduces to~\eqref{eq: rigidity}.) We point out that the rigidity estimate~\eqref{eq: rigidity} is a nonlinear version of Korn's inequality, which allows to estimate from below the $L^{p}$-norm of the symmetrized gradient (i.e., the distance from infinitesimal rotations) with the $L^{p}$-distance of the whole gradient from a single skew-symmetric matrix. Also Korn's inequality has been generalized to the mixed growth setting and is indeed a crucial ingredient to prove~\eqref{eq:rigidity2}, see~\cite{CDM}.

From an applicative perspective, estimate~\eqref{eq: rigidity} has proved to play a pivotal role in the analysis of variational models in nonlinear elasticity, for it delivers compactness for sequences of deformations and corresponding displacements with uniformly bounded elastic energy. In fact, it is the cornerstone for rigorous derivations of lower dimensional theories for plates, shells, and rods in various scaling regimes~\cite{FJM, hierarchy, Lewicka, Mora4, Mora, Mora2}, and for providing relations between geometrically nonlinear and linear models in elasticity~\cite{DalMasoNegriPercivale:02}. In~\cite{AgDMDS,Mora-Riva}, the version with mixed exponents~\eqref{eq:rigidity2} is used to prove strong convergence of recovery sequences.

The estimate~\eqref{eq: rigidity} was generalized in various directions. Without being exhaustive, we mention results for incompatible fields~\cite{conti.garroni, lauteri.luckhaus, Muller-Scardia-Zeppieri:14}, settings involving multiple energy wells~\cite{Chaudhuri, Chermisi-Conti, conti.schweizer, davoli.friedrich, De Lellis, Jerrard-Lorent, Lorent}, and estimates in the realm of free-discontinuity problems 
~\cite{Chambolle-Giacomini-Ponsiglione:2007, KFZ:2021, Friedrich-Schmidt:15} where deformations can exhibit discontinuities. Based on these results, among others, the studies on linearization of nonlinearly elastic energies have been successfully extended in several ways, including incompressible materials~\cite{Jesenko-Schmidt:20, edo}, the passage from atomistic-to-continuum models~\cite{Braides-Solci-Vitali:07, KFZ:2022, Schmidt:2009}, multiwell energies~\cite{alicandro.dalmaso.lazzaroni.palombaro, davoli2, Schmidt:08}, plasticity~\cite{Ulisse}, thermoviscoelasticity~\cite{Rufat, MFMK}, fracture~\cite{Friedrich:15-2, higherordergriffith}, or elastic thin films~\cite{KFZ:2021, KrePio19}.

 The goal of this article is to extend the estimates~\eqref{eq: rigidity}--\eqref{eq:rigidity2} to the setting of variable exponents and to generalize results on the rigorous connection between nonlinear and linearized models in elasticity theory to this framework. Therefore, our analysis has both a mathematical and a modeling interest. 

From a theoretical standpoint, Korn's inequality has already been extended to the variable exponent setting~\cite{DHHR} and it is thus a natural question if a version with mixed variable growth holds and if it can be used to prove a rigidity estimate in the Sobolev space $W^{1,p(x)}$. In particular, the latter may still be interpreted as a quantitative version of {\sc Liouville}'s and {\sc Reshetnyak}'s rigidity results, which are then extended to the variable exponent framework. 

From a modeling point of view, variable exponents are customary in continuum mechanics for describing the behaviors of fluids, which change their mechanical response according to an external electric field, see \cite{Diening} for further details. This is a point of view which may be also reasonable for elastic materials. Furthemore, variable growth conditions can also account for heterogenity.  In this context, we remark that a large class of compressible rubber-like materials are described in terms of energy densities growing quadratically close to the energy wells and less than quadratically far away from the wells. The qualitative description of such materials, however, usually relies on linearized approximations, whose range of validity has to be understood. The passage from a nonlinear to a linearized model for a fixed exponent~$1 < p \leq 2$ was considered in~\cite{AgDMDS}, while a variable exponent $1 < p(x) \leq 2$ can further capture the behavior of composite materials characterized by a strong anisotropy~\cite{zikov3, zikov}. In the framework of $\Gamma$-convergence~\cite{DalMaso93AnIntroduction}, when passing from a nonlinear elastic energy to its linear counterpart, one aims at proving the convergence of minimizers of the nonlinear energies to the minimizers of the linearized limit model in a suitable Sobolev space. By comparison with the arguments of~\cite{AgDMDS} (see also~\cite{Schmidt:08}) for the fixed exponent setting, geometric rigidity estimates are expected to be vital to obtain convergence in~$W^{1,p(x)}$. 

In order to better explain the role played by the variable exponent $p(\cdot)$, we notice that, for a fixed exponent~$p$, inequality~\eqref{eq: rigidity} can be equivalently written as an inequality between norms
\begin{equation}
\label{eq: rigidity-3}
\| \nabla y - R \|_{p } \le C \big\| \dist \big(\nabla y,SO(n)\big) \big\|_{p}\,.
\end{equation}
In the Sobolev space $W^{1,p(x)}(\Omega)$, instead, it is important to decide whether to express the rigidity estimate in terms of the norm $\| \cdot\|_{p(x)}$ or of the modular~$\int_{\Omega}|\cdot|^{p(x)}$. Indeed, while inducing the same topologies, the two quantities cannot be mutually controlled, as the modular lacks of homogeneity. In particular, some basic functional inequalities simply do not hold for the modular while being true for the norm~$\| \cdot\|_{p(x)}$. In this respect, we refer for instance to~\cite[Example 8.2.7]{DHHR} for an explicit counterexample to Poincar\'e's inequality. In the same spirit, we notice that many fundamental tools necessary in the analysis of Sobolev spaces with variable exponent, such as estimates on the maximal operator, regularity of elliptic partial differential equations, and Korn's inequalities, are expressed by means of the norm~\cite{DHHR}. Thus, it is natural to express rigidity in terms of norms~\eqref{eq: rigidity-3} rather than for the modular~\eqref{eq: rigidity}. Moreover, due to the one-homogeneity of the norm, a rigidity estimate in $W^{1,p(x)}$ expressed in the form~\eqref{eq: rigidity-3} is convenient for the study of convergence of minima and minimizers in the linearization process, where deformations are multiplied by a singular prefactor~$\frac{1}{\varepsilon}$.

The strategy of the proof of our rigidity estimates follows closely the classical ones of~\cite{FJM} and of~\cite{CDM}. It is based on a covering argument, where the required estimate is first recovered on small cubes and then extended to the whole domain. The adaptation of this proof strategy to the variable exponent setting presents several technical issues. First, a uniform control (with respect to the size of the cube) of the constants appearing in the local rigidity estimate on small cubes (Proposition~\ref{prop:Qpxrig}) is crucial to derive the global result. Secondly, some technical tools have to be specified to our more general setting. This is for instance the case of an extension result in~$W^{1,p(x)}$ for functions with mixed growth in the symmetric gradient (see Theorem~\ref{thm:extension}), where {\sc Nitsche}'s techniques~\cite{nitsche} are used to find a suitable extension of the exponent as well. We also remark that the classical Lusin approximation argument has to be formulated and used in a slightly different way than in~\cite{FJM}. In our case, continuity of the maximal operator is explicitly invoked to estimate the $L^{p(\cdot)}$-norm of the exceptional set, which is indeed a sublevel set of the maximal function. In all these aspects, a fundamental role is played by a suitable continuity assumption on the exponent.

The above mentioned key assumption on~$p(\cdot)$ is the so-called {\em log-H\"older continuity} (see, e.g.,~\cite{DHHR}), which provides~$p(\cdot)$ with an explicit modulus of continuity and is particularly fit for blow-up and localization methods. Such condition is by now customary in the variable exponent framework. Besides the reference book~\cite{DHHR}, we mention~\cite{zikov2}, where the log-H\"older continuity of the exponent prevents from the Lavrentiev's phenomenon, the works~\cite{acerbi2, acerbi3, MR3209686, MR2391646} for the regularity of minimizers of functionals with~$p(\cdot)$-growth and of solutions to elliptic partial differential equations with general growth, and~\cite{Alm-Reg-Sol_23, DCLVlscvariable, Leo-Sci-Sol-Ver_23, Sci-Sol-Str_22} for integral representation, lower-semicontinuity, $\Gamma$-convergence, and regularity issues for free discontinuity functionals. 

In our setting, the log-H\"older continuity is essential to keep track of and suitably control the constants appearing in our arguments. When proving the rigidity estimate in Section~\ref{sec:2}, we first provide a localized version of the desired inequality and then apply a local-to-global argument (cf.~Proposition~\ref{prop:locglob}) which indeed builds upon the log-H\"older condition. In Section~\ref{sec:3}, with similar ideas, we prove a Korn's inequality with mixed variable growth, which then leads to the generalization of~\eqref{eq:rigidity2} to the variable exponent case. We refer to the proof of Theorem~\ref{thm:pxqxkorn} for full details. A further technical ingredient needed in both our rigidity estimates is a Lusin approximation (see Lemma~\ref{lem:lusin}), which must be reformulated in a slightly different form compared to~\cite{FJM} due to the non-equivalence of norm and modular in~$W^{1,p(x)}$. In particular, we make use of the Lusin approximation in combination with continuity estimates for the maximal operator in Lebesgue spaces with variable exponents, which again hold under the log-H\"older continuity of~$p(\cdot)$. We refer to Remark~\ref{rem:on condition (ii)} and to the proofs of Proposition~\ref{prop:Q0pxrig} and of Theorem~\ref{thm:pxqxrig} for a complete discussion. 

The main results of our paper are the rigidity estimate for variable exponents (Theorem~\ref{thm:pxrig}), a Korn inequality, and a rigidity estimate with mixed variable growth (Theorems~\ref{thm:pxqxkorn} and~\ref{thm:pxqxrig}, respectively). In Section~\ref{sec:4}, as an application, we generalize the results of~\cite{AgDMDS, DalMasoNegriPercivale:02} on the passage from nonlinear to linearized elasticity with variable subquadratic growth far from the energy well. In this regard, we notice that the $\Gamma$-convergence stated in Theorem~\ref{thm:gamma_con} does not need the regularity of the exponent~$p(\cdot)$, as compactness and the construction of a recovery sequence rely on a weaker rigidity for the modular (see Theorem~\ref{thm:gpxrig}) and on the arguments of~\cite{AgDMDS}. This is important from an applicative point of view as it allows for the modeling of composite materials. The rigidity results of Theorems~\ref{thm:pxrig} and~\ref{thm:pxqxrig} are instead crucial to prove strong convergence of minimizers in~$W^{1,p(x)}$. 

We close the introduction with some final remarks. Firstly, in order to emphasize the role played by the variable exponent, in our presentation we have gone for a higher level of detail than other works on related themes, even if some of our arguments are quite standard by now. Secondly, we believe that our results can be the starting point for more sophisticated rigidity estimates. Indeed, in~\cite{CDM}, the mixed growth rigidity for fixed exponents can be generalized to estimates in Lorentz spaces. This in turn is a key tool for proving a version for incompatible fields~\cite{Muller-Scardia-Zeppieri:14} leading to numerous applications in strain-gradient plasticity. Extensions to the setting of variable exponents will be subject of future research. 

%-----------------------------------------------------
% Lebesgue and Sobolev spaces with variable exponent
%-----------------------------------------------------

\section{Notations and preliminary results}

\subsection{Notation}

The space of $d\times n$ matrices with real entries is denoted by $\R^{d\times n}$. Given two matrices $A_1,A_2\in \R^{d\times n}$, their scalar product is denoted by $A_1\colon A_2$ and the induced norm of $A\in\R^{d\times n}$ by $|A|$. In the case $d=n$, the subspace of symmetric matrices is denoted by $\R^{n\times n}_{\sym}$, and the subspace of skew-symmetric matrices by $\R^{n\times n}_{\skw}$. Given $A\in\R^{n\times n}$, we denote by $A_{\sym}:=\frac{1}{2}(A+A^T)\in\R^{n\times n}_{\sym}$ its symmetric part, and by $A_{\skw}:=A-A_{\sym}=\frac{1}{2}(A-A^T)\in\R^{n\times n}_{\skw}$ its skew-symmetric part. We use $SO(n)$ to denote the special orthogonal group in $\R^n$, consisting of all matrices $A\in\R^{n\times n}$ satisfying $A^{-1}=A^T$ and $\det A=1$. We denote by $d(A, SO(n))$ the Euclidean distance of $A\in\R^{n\times n}$ from $SO(n)$. The identity matrix is denoted by $I \in \R^{n \times n}$. 

For a measurable set $E \subseteq \R^n$ we use $|E|$ to denote the $n$-dimensional Lebesgue measure of $E$. By $E^c := \R^n\setminus E$ we indicate its complement and by $\diam(E)$ its diameter. By $\chi_E \colon E \to \{ 0 ,1 \}$ we denote the corresponding characteristic function. An open and connected set $\Omega\subseteq\R^n$ is called domain. 

Given an open subset $\Omega$ of $\R^n$, the set of all distributions on $\Omega$, namely the continuous dual space of $C_c^\infty (\Omega;\R^d)$, endowed with the strong dual topology, is denoted by $\mathcal{D}'(\Omega;\R^d)$. We adopt standard notation for Lebesgue spaces on measurable subsets $E\subseteq\mathbb R^n$ and Sobolev spaces on open subsets $\Omega\subseteq\mathbb R^n$. Given $u\in L^1(\Omega;\R^d)$ we denote by $\langle u\rangle_\Omega\in\R^d$ its mean value on $\Omega$, i.e., $\langle u\rangle_\Omega:=\frac{1}{|\Omega|}\int_\Omega u(x)\,\de x$. According to the context, we use $\|\cdot\|_{L^p(E)}$ to denote the norm in $L^p(E;\R^d)$ for every $1\le p\le\infty$ and $d\in\N$. A similar convention is also used to denote the norms in Sobolev spaces. The boundary values of a Sobolev function are always intended in the sense of traces.

The partial derivatives with respect to the variable $x_i$ are denoted by $\partial_i$. Given an open subset $\Omega\subseteq\mathbb R^n$ and a function $u\colon\Omega\to\R^d$, we denote its Jacobian matrix by $\nabla u$, whose components are $(\nabla u)_{ij}:=\partial_j u_i$ for $i=1,\dots,d$ and $j=1,\dots,n$. We set $\nabla^2 u:=\nabla(\nabla u)$ and we use $\Delta u$ to denote the Laplacian of $u$, which is defined as $\Delta u:=\sum_{i=1}^d\partial^2_{ii} u $. For a function $u\colon \Omega \to \R^n$ we use $eu:= (\nabla u)_{\sym}$ to denote the symmetric part of the gradient. Given a tensor field $F\colon \Omega\to\R^{d\times n}$, by $\div F$ we mean its divergence with respect to lines, namely $(\div F)_i:=\sum_{j=1}^n\partial_jF_{ij}$ for $i=1,\dots, d $. 

We use the convention that constants may change from line to line. We will frequently emphasize the explicit dependence of the constants on the parameters for the sake of clarity. 

\subsection{Lebesgue and Sobolev spaces with variable exponent and their properties}

In the following, we introduce the notions of Lebesgue and Sobolev spaces with variable exponent, and we state the main properties that will be used throughout the paper. For more information regarding these spaces, we refer to the book~\cite{DHHR} and the references therein. 

Let $E\subseteq\R^n$ be a measurable set and let $p\colon E\to \R$ be a measurable function. For all measurable subsets $F\subseteq E$ we define
$$p^-_F:= \essinf_{x\in F} p(x)\quad\text{and}\quad p^+_F:=\esssup_{x\in F}p(x),$$
and in the case $F=E$ we simply write $p^-$ and $p^+$. We set
$$\mathcal P_b(E):=\{p\colon E\to[1,\infty)\,:\,\text{$p$ is measurable with $p^+<\infty$}\}.$$ 
Given a function $p\in\mathcal P_b(E)$, 
the Lebesgue space with variable exponent $p$ is defined as
$$L^{p(\cdot)}(E;\R^d):=\left\{f\colon E\to\R^d\text{ measurable}\,:\,\int_E|f(x)|^{p(x)}\,\de x<\infty\right\}.$$
This is a Banach space endowed with the norm
$$\|f\|_{L^{p(\cdot)}(E)}:=\inf\left\{\lambda>0\,:\,\int_{E}\left|\frac{f(x)}{\lambda}\right|^{p(x)}\,\de x\le 1\right\}.$$
By definition, the norm $\|\cdot\|_{L^{p(\cdot)}(E)}$ satisfies the following properties:
\begin{align*}
&\|f+g\|_{L^{p(\cdot)}(E)}\le \|f\|_{L^{p(\cdot)}(E)}+\|g\|_{L^{p(\cdot)}(E)}\quad\text{for all $f,g\in L^{p(\cdot)}(E;\R^d)$},\\
&\|f\|_{L^{p(\cdot)}(F)}=\|f\chi_F\|_{L^{p(\cdot)}(E)}\quad\text{for all measurable sets $F\subseteq E$ and $f\in L^{p(\cdot)}(E;\R^d)$},\\
&\||f|^s\|_{L^{p(\cdot)}(E)}=\|f\|_{L^{sp(\cdot)}(E)}^s\quad\text{for all $s\in [0,\infty)$ with $sp\in \mathcal P_b(E)$ and $f\in L^{sp(\cdot)}(E;\R^d)$},\\
&\|f\|_{L^{p(\cdot)}(E)}\le \|g\|_{L^{p(\cdot)}(E)}\quad\text{for all $f,g\in L^{p(\cdot)}(E;\R^d)$ with $|f|\le |g|$ a.e.\ in $E$}.
\end{align*}
The quantity $\int_E|f(x)|^{p(x)}\,\mathrm{d}x$ is often referred to as the \emph{modular} of $f$ and denoted by $\varrho_{p(\cdot)}(f)$. We recall the following results for $L^{p(\cdot)}(E;\R^d)$, see for example~\cite{DHHR}.

\begin{proposition}[{\cite[Lemma~3.2.5 and Lemma~3.4.2]{DHHR}}]\label{prop:modular}
Let $p\in\mathcal P_b(E)$ and let $f\in L^{p(\cdot)}(E;\R^d)$. 
\begin{itemize}
\item[(i)] If $\|f\|_{L^{p(\cdot)}(E)}\le 1$, then
$$\|f\|_{L^{p(\cdot)}(E)}^{p^+}\le \int_{E}|f(x)|^{p(x)}\,\de x\le \|f\|_{L^{p(\cdot)}(E)}^{p-}.$$
\item [(ii)] If $\|f\|_{L^{p(\cdot)}(E)}\ge 1$, then
$$ \|f\|_{L^{p(\cdot)}(E)}^{p^-} \le \int_{E}|f(x)|^{p(x)}\,\de x\le \|f\|_{L^{p(\cdot)}(E)}^{p^+}.$$
\end{itemize}
In particular,
$$\|f\|_{L^{p(\cdot)}(E)}\le 1\quad\Longleftrightarrow\quad\int_{E}\left|f(x)\right|^{p(x)}\,\de x\le 1,\quad \|f\|_{L^{p(\cdot)}(E)}=1\quad\Longleftrightarrow\quad\int_{E}|f(x)|^{p(x)}\,\de x= 1.$$
\end{proposition}

\begin{remark}\label{rem:volume}
As an immediate consequence of Proposition~\ref{prop:modular} applied for $f \equiv 1$ on $E$, we get
$$ |E|^{\frac{1}{p^-}} \le \|1\|_{L^{p(\cdot)}(E)} \le |E|^{\frac{1}{p^+}} \quad \text{if $|E| \le 1$}, \quad \quad \text{and} \quad \quad |E|^{\frac{1}{p^+}} \le \|1\|_{L^{p(\cdot)}(E)} \le |E|^{\frac{1}{p^-}} \quad \text{if $|E|\ge 1$}. $$
\end{remark}

For every $\lambda>0$ and $x_0\in \R^n$ we define $\psi_{x_0,\lambda}\colon \R^n\to \R^n$ as 
\begin{equation}\label{eq:psilambda}
\psi_{x_0,\lambda}(x):=x_0+\lambda x\quad\text{for $x\in\R^n$}. 
\end{equation}
Clearly, $\psi_{x_0,\lambda}$ is invertible and it inverse is given by 
$$\psi_{x_0,\lambda}^{-1}(y)=\frac{y-x_0}{\lambda}\quad\text{for $y\in\R^n$}.$$
By Proposition~\ref{prop:modular} and the area formula we can easily derive the following result.

\begin{proposition}\label{prop:change}
Let $p\in\mathcal P_b(E)$ and let $f\in L^{p(\cdot)}(E;\R^d)$. For $\lambda>0$ and $x_0\in \R^n$ define 
$$g:=f\circ \psi_{x_0,\lambda}\quad\text{and}\quad q:=p\circ\psi_{x_0,\lambda}\quad\text{in $\psi^{-1}_{x_0,\lambda}(E)$}.$$ Then, $q\in\mathcal P_b(\psi_{x_0,\lambda}^{-1}(E))$, $g\in L^{q(\cdot)}(\psi_{x_0,\lambda}^{-1}(E))$, and 
\begin{equation*}
\min\{\lambda^{-\frac{n}{p^-}},\lambda^{-\frac{n}{p^+}}\} \|f\|_{L^{p(\cdot)}(E)}\le\|g\|_{L^{q(\cdot)}(\psi_{x_0,\lambda}^{-1}(E))}\le \max\{\lambda^{-\frac{n}{p^-}},\lambda^{-\frac{n}{p^+}}\} \|f\|_{L^{p(\cdot)}(E)}.
\end{equation*}
\end{proposition}

\begin{proposition}[{\cite[Lemma~3.2.20]{DHHR}}]\label{prop:spq}
Let $p,q,s\in\mathcal P_b(E)$ be satisfying
$$\frac{1}{s(x)}=\frac{1}{p(x)}+\frac{1}{q(x)}\quad\text{for a.e.\ $x\in E$}.$$
Let $f\in L^{p(\cdot)}(E;\R^d)$ and $g\in L^{q(\cdot)}(E;\R^d)$. Then $f\cdot g\in L^{s(\cdot)}(E)$ and
$$\|f\cdot g\|_{L^{s(\cdot)}(E)}\le 2\|f\|_{L^{p(\cdot)}(E)}\|g\|_{L^{q(\cdot)}(E)}.$$
\end{proposition}
 In particular, for $s \equiv 1$, the exponent $p' := q $ satisfies $1=\frac{1}{p(x)}+\frac{1}{p'(x)}$ for a.e.\ $x\in E$, and is called the \emph{dual variable exponent} of $p$.

\begin{proposition}[{\cite[Corollary~3.3.4]{DHHR}}]
Assume that $|E|<\infty$. Let $p,q\in\mathcal P_b(E)$ be satisfying 
$$p(x)\le q(x)\quad\text{for a.e.\ $x\in E$}.$$
Then 
$$L^{q(\cdot)}(E;\R^d)\subseteq L^{p(\cdot)}(E;\R^d)$$
with continuous embeddings.
\end{proposition}

\begin{proposition}[{\cite[Lemma~3.4.4 and Theorem~3.4.7]{DHHR}}]
Let $p\in\mathcal P_b(E)$ with $p^->1$. The Banach space $L^{p(\cdot)}(E;\R^d)$ is separable and reflexive.
\end{proposition}

From now on, let $\Omega\subseteq\R^n$ be an open set and let $p\in\mathcal P_b(\Omega)$. The Sobolev space with variable exponent $p$ is defined as
$$W^{1,p(\cdot)}(\Omega;\R^d):=\left\{f\in L^{p(\cdot)}(\Omega;\R^d)\,:\, \nabla f\in L^{p(\cdot)}(\Omega;\R^{d\times n})\right\}.$$

We recall the following results for $W^{1,p(\cdot)}(\Omega;\R^d)$.

\begin{proposition}[{\cite[Corollary~3.3.4]{DHHR}}]
Let $|\Omega|<\infty$ and let $p,q\in\mathcal P_b(\Omega)$ be satisfying 
$$p(x)\le q(x)\quad\text{for a.e.\ $x\in\Omega$}.$$
Then 
$$W^{1,q(\cdot)}(\Omega;\R^d)\subseteq W^{1,p(\cdot)}(\Omega;\R^d)$$
with continuous embeddings.
\end{proposition}

\begin{proposition}[{\cite[Theorem~8.1.6]{DHHR}}]
Let $p\in\mathcal P_b(\Omega)$ with $p^->1$. Then the Banach space $W^{1,p(\cdot)}(\Omega;\R^d)$ is separable and reflexive.
\end{proposition}

We now introduce the $\log$-Hölder condition for a variable exponent $p$, which is needed to gain additional properties for the Sobolev space with variable exponent $p$. 

A function $\alpha\colon \Omega\to\R$ is \emph{locally $\log$-Hölder continuous} on $\Omega$ if there exists a constant $c_1(\alpha)\ge 0$ such that
\begin{equation}\label{eq:logcon}
|\alpha(x)-\alpha(y)|\le \frac{c_1(\alpha)}{\log(\e + 1/|x-y|)}\quad\text{for every $x,y\in \Omega$ with $x\neq y$}.
\end{equation}
A function $\alpha$ satisfies the \emph{$\log$-Hölder decay condition} if there exist $\alpha_\infty\in\R$ and a constant $c_2(\alpha)\ge 0$ such that
\begin{equation}\label{eq:logcon2}
|\alpha(x)-\alpha_\infty|\le \frac{c_2(\alpha)}{\log(\e + |x|)}\quad\text{for every $x,y\in \Omega$}.
\end{equation}
A function $\alpha$ is \emph{globally $\log$-Hölder continuous} on $\Omega$ if it is locally $\log$-Hölder continuous on $\Omega$ and satisfies the $\log$-Hölder decay condition. The constants $c_1(\alpha)$ and $c_2(\alpha)$ are called the local $\log$-Hölder constant and the $\log$-Hölder decay constant, respectively. The maximum $\max\{c_1(\alpha),c_2(\alpha)\}$ is just called the \emph{$\log$-Hölder constant} of $\alpha$ and is denoted by $c_{\log}(\alpha)$.

\begin{remark}\label{rem:H_logglob}
When $\Omega$ is a bounded open set, we can simplify the notion of globally $\log$-Hölder continuity. In this case, a function $\alpha\colon\Omega\to\R$ is globally $\log$-Hölder continuous if and only if it is locally $\log$-Hölder continuous and $c_{\log}(\alpha)$ depends only on $c_1(\alpha)$, $\alpha^-$, $\alpha^+$, and the size of $\Omega$. Equivalently, a function $\alpha\colon\Omega\to\R$ is globally $\log$-Hölder continuous if and only if there exists a constant $c_3(\alpha)>0$ such that
$$|\alpha(x)-\alpha(y)|\le \frac{c_3(\alpha)}{-\log|x-y|}\quad\text{for all $x,y\in\Omega$ with $0<|x-y|\le \frac{1}{2}$},$$
which is often used in the literature for the definition of the (locally) $\log$-Hölder condition.
\end{remark}

The geometrical meaning of the $\log$-Hölder condition~\eqref{eq:logcon} is given by the following proposition.

\begin{proposition}[{\cite[Lemma~4.1.6]{DHHR}}]\label{prop:smallQ}
Let $\alpha\colon\R^n\to\R$ be a continuous and bounded function. The following are equivalent
\begin{itemize}
\item[(i)] $\alpha$ is locally $\log$-H\"older continuous;
\item[(ii)] there exists a constant $C>0$ such that for all cubes $Q$ we have $|Q|^{\alpha_Q^--\alpha_Q^+}\le C$.
\end{itemize}
The constant $C$ depends on $n$ and $c_1(\alpha)$.
\end{proposition}

In what follows, let $\Omega\subseteq\R^n$ be an open set and define
$$\mathcal P^{\log}_b(\Omega):=\left\{p\in\mathcal P_b(\Omega)\,:\,\text{$p$ is globally $\log$-Hölder continuous}\right\}.$$
 As $p^+<\infty$, due to~\cite[Remark 4.1.5]{DHHR}, we notice that the space $\mathcal P^{\log}_b(\Omega)$ coincides with the one defined in~\cite[Definition 4.1.4]{DHHR}. Moreover, it is easy to check that $p\in\mathcal P^{\log}_b(\Omega)$ and $p^->1$ if and only if $p'\in\mathcal P^{\log}_b(\Omega)$ and $(p')^->1$, where $p'$ is defined below Proposition~\ref{prop:spq}. 

Sometimes it is convenient to deal with exponents $p$ which are defined in the entire space $\R^n$. This can be done thanks to the following proposition. 

\begin{proposition}[{\cite[Proposition~4.1.7]{DHHR}}]\label{prop:417}
Let $p\in\mathcal P^{\log}_b(\Omega)$. There exists $q\in\mathcal P^{\log}_b(\R^n)$ such that
$$q|_\Omega=p,\quad q^-=p^-,\quad q^+=p^+,\quad c_{\log}(q)=c_{\log}(p).$$
\end{proposition}

We now recall some functional analytic estimates for $\log$-Hölder variable exponents, that will be used in the paper. The first two results are a continuity estimate for the maximal operator. Given a function $f\in L^1_{\loc}(\R^n;\R^d)$, we define the \emph{maximal function} $M(f)\colon\R^n\to\R$ of $f$ as
\begin{equation}\label{eq:maxfun}
M(f)(x):=\sup_{\rho>0}\frac{1}{|B_{\rho}(x)|}\int_{B_{\rho}(x)}|f(y)|\,\de y\quad\text{for all $x\in\R^n$}.
\end{equation}

\begin{proposition}[{\cite[Theorem~4.3.8]{DHHR}}]\label{prop:maxfun}
Let $p\in\mathcal P^{\log}_b(\R^n)$ with $p^->1$. There exists a constant $C=C(n,p)>0$ such that for all $f\in L^{p(\cdot)}(\R^n;\R^d)$
$$\|M(f)\|_{L^{p(\cdot)}(\R^n)}\le C\|f\|_{L^{p(\cdot)}(\R^n)}.$$
\end{proposition}

Let now $\Omega\subset\R^n$ be a bounded domain with Lipschitz boundary. Given a function $f\in L^1(\Omega;\R^d)$, we define the \emph{local maximal function} $M_\Omega(f)\colon \Omega\to\R$ of $f$ as
$$M_\Omega(f)(x):=\sup_{\rho>0}\frac{1}{|\Omega\cap B_{\rho}(x)|}\int_{\Omega\cap B_{\rho}(x)}|f(y)|\,\de y\quad\text{for all $x\in \Omega$}.$$
We extend $f$ to the entire $\R^n$ by setting $f = 0$ outside $\Omega$ and we consider the maximal function $M(f)\colon\R^n\to\R$ defined in~\eqref{eq:maxfun}. We have
\begin{align*}
&M(f)(x)=\sup_{\rho\in (0,\diam(\Omega)]}\frac{1}{|B_{\rho}(x)|}\int_{\Omega\cap B_{\rho}(x)}|f(y)|\,\de y,\\
&M_\Omega(f)(x)=\sup_{\rho\in (0,\diam(\Omega)]}\frac{1}{|\Omega\cap B_{\rho}(x)|}\int_{\Omega\cap B_{\rho}(x)}|f(y)|\,\de y,
\end{align*}
as $\Omega\subset B_\rho(x)$ for all $x\in \Omega$ and $\rho>\diam(\Omega)$. Since $\Omega$ is a bounded domain with Lipschitz boundary, there exists a constant $\sigma=\sigma(\Omega,n)>0$ such that
\begin{equation}\label{eq:areaball}
|B_\rho(x)|\le \sigma|\Omega\cap B_{\rho}(x)|\quad\text{for all $x\in\Omega$ and $\rho\in (0,\diam(\Omega)]$},
\end{equation}
see for example~\cite[Chapter 5.1]{GM}. Hence, in view of~\eqref{eq:areaball}, 
\begin{equation}\label{eq:Omaxfun}
M_\Omega(f)(x)\le \sigma M(f)(x)\quad\text{for all $x\in \Omega$}.
\end{equation}
As a consequence of Propositions~\ref{prop:417}--\ref{prop:maxfun} and~\eqref{eq:Omaxfun}, we have the following. 

\begin{proposition}\label{prop:maxfun2}
Let $\Omega\subset\R^n$ be a bounded domain with Lipschitz boundary and let $p\in\mathcal P^{\log}_b(\Omega)$ with $p^->1$. There exists a constant $C=C(\Omega, n,p)>0$ such that for all $f\in L^{p(\cdot)}(\Omega;\R^d)$
$$\|M_\Omega(f)\|_{L^{p(\cdot)}(\Omega)}\le C\|f\|_{L^{p(\cdot)}(\Omega)}.$$
\end{proposition}

The third result is a version of the second Korn inequality for variable exponent.

\begin{proposition}[{\cite[Theorem~14.3.23]{DHHR}}]\label{prop:pxkorn}
Let $\Omega\subset\R^n$ be a bounded domain with Lipschitz boundary. Let $p\in\mathcal P^{\log}_b(\Omega)$ with $p^->1$. There exists a constant $C=C(\Omega,n,p)>0$ such that for all $u\in W^{1,p(\cdot)}(\Omega;\R^n)$ there exists a skew-symmetric matrix $S\in \R^{n\times n}_{\skw}$ satisfying
\begin{equation*} 
\|\nabla u-S\|_{L^{p(\cdot)}(\Omega)}\le C\|eu\|_{L^{p(\cdot)}(\Omega)}.
\end{equation*}
In particular, we can take $S=(\langle\nabla u\rangle_\Omega)_{\skw}\in\R^{n\times n}_{\skw}$.
\end{proposition}

The fourth preliminary result concerns elliptic estimates for the solutions to Poisson problems with variable exponents.

\begin{proposition}[{\cite[Theorem~14.1.2]{DHHR}}]\label{prop:pxpoisson}
Let $\Omega\subset\R^n$ be a bounded domain with $C^{1,1}$ boundary. Let $p\in\mathcal P^{\log}_b(\Omega)$ with $p^->1$. For all functions $f\in L^{p(\cdot)}(\Omega;\R^d)$ there exists a unique strong solution $u\in W^{2,p(\cdot)}(\Omega;\R^d)$ to the problem
$$\begin{cases}
-\Delta u=f &\quad\text{in $\Omega$},\\
u=0 &\quad\text{on $\partial\Omega$}.
\end{cases}$$
Moreover, there exists a constant $C=C(\Omega,n,d,p)>0$ such that
\begin{equation}\label{eq:ell-est}
\|u\|_{W^{2,p(\cdot)}(\Omega)}\le C\|f\|_{L^{p(\cdot)}(\Omega)}.
\end{equation}
\end{proposition}

 Finally, we recall the following localization techniques which are customary when dealing with variable exponents. A family $\mathcal Q$ of cubes $Q\subset\R^n$ is called locally $N$-finite, $N\in\N$, if it satisfies
$$\sum_{Q\in\mathcal Q}\chi_Q\le N\quad\text{a.e.\ in $\R^n$}.$$

\begin{proposition}[{\cite[Corollary~7.3.21]{DHHR}}]\label{prop:locglob}
 Let $p\in\mathcal P^{\log}_b(\R^n)$ and let $\mathcal Q$ be a family of locally $N$-finite cubes $Q\subset\R^n$. There exists a constant $C=C(n,p,N) \ge 1 $ such that for all $f\in L^{p(\cdot)}(\R^n)$
\begin{equation}\label{eq:locglob}
\frac{1}{C}\bigg\|\sum_{Q\in\mathcal Q}\chi_Q f\bigg\|_{L^{p(\cdot)}(\R^n)}\le\bigg\|\sum_{Q\in\mathcal Q}\chi_Q \frac{\|\chi_Q f\|_{L^{p(\cdot)}(\R^n)}}{\|\chi_Q\|_{L^{p(\cdot)}(\R^n)}}\bigg\|_{L^{p(\cdot)}(\R^n)}\le C\bigg\|\sum_{Q\in\mathcal Q}\chi_Q f\bigg\|_{L^{p(\cdot)}(\R^n)}. 
\end{equation}
\end{proposition}

\begin{proposition}[{\cite[Corollary~7.3.24]{DHHR}}]\label{prop:locglob2}
Let $p\in\mathcal P^{\log}_b(\R^n)$, let $\mathcal Q$ be a family of locally $N$-finite cubes $Q\subset\R^n$, and let $\{f_Q\}_{Q\in\mathcal Q}$ be a family of functions satisfying $f_Q\in L^{p(\cdot)}(Q)$ for all $Q\in\mathcal Q$. There exists a constant $C=C(n,p,N) \ge 1 $ such that for all $f\in L^{p(\cdot)}(\R^n)$
\begin{equation}\label{eq:locglob2}
\bigg\|\sum_{Q\in\mathcal Q}\chi_Q f_Q\bigg\|_{L^{p(\cdot)}(\R^n)}\le C\bigg\|\sum_{Q\in\mathcal Q}\chi_Q \frac{\|\chi_Q f_Q\|_{L^{p(\cdot)}(\R^n)}}{\|\chi_Q\|_{L^{p(\cdot)}(\R^n)}}\bigg\|_{L^{p(\cdot)}(\R^n)}. 
\end{equation}
\end{proposition}

\begin{remark}[Dependence on $p$]\label{rem:p}
 In Propositions~\ref{prop:maxfun}--\ref{prop:locglob2}, the constant $C$ depends on $p$ only via $p^-$, $p^+$, and $c_{\rm log}(p)$, see~\cite{DHHR}.
\end{remark}

%---------------------------------------------------
% Rigidity on Sobolev spaces with variable exponent
%---------------------------------------------------

\section{Geometric rigidity on Sobolev spaces with variable exponent}\label{sec:2}

The goal of this section is to extend the rigidity result of~\cite[Theorem 3.1]{FJM} to the case of a variable exponent $p$. More precisely, we prove the following theorem.

\begin{theorem}[Geometric rigidity for variable exponents]\label{thm:pxrig}
Let $\Omega\subset\R^n$ be a bounded domain with Lipschitz boundary. Let $p\in\mathcal P^{\log}_b(\Omega)$ with $p^->1$. There exists a constant $C(\Omega,n,p)>0$ such that for all $u\in W^{1,p(\cdot)}(\Omega;\R^n)$ we can find a constant rotation $R\in SO(n)$ satisfying
\begin{equation}\label{eq:pxrig}
\|\nabla u-R\|_{L^{p(\cdot)}(\Omega)}\le C\|d(\nabla u,SO(n))\|_{L^{p(\cdot)}(\Omega)}.
\end{equation}
 In particular, we can take $R\in SO(n)$ such that 
\begin{equation}\label{eq:R-mean}
|R-\langle\nabla u\rangle_\Omega|=d(\langle\nabla u\rangle_\Omega,SO(n)).
\end{equation}
\end{theorem}

The proof of Theorem~\ref{thm:pxrig} follows the same structure of~\cite{FJM} and it proceeds according to the following steps:
\begin{itemize}
\item In Proposition~\ref{prop:Q0pxrig} we first prove the result locally in cubes $Q$. This relies on the Lusin approximation of Lemma~\ref{lem:lusin} and on the rigidity result in $W^{1,p^-}(Q;\R^n)$. 
\item Then in Proposition~\ref{prop:Qpxrig} we show that the rigidity constant can be taken independently of the size of the cube. This makes use of the $\log$-Hölder condition. 
\item Afterwards, we combine a weighted Poincaré inequality (Proposition~\ref{prop:poincare}) with Proposition~\ref{prop:locglob} to pass from cubes $Q$ to every bounded domain $\Omega$ with Lipschitz boundary.
\end{itemize}

The first step relies on a classical truncation argument, which is the Lusin approximation. We state in a form that is suitable for our purposes.

\begin{lemma}[Lusin approximation]\label{lem:lusin}
Let $\Omega\subset\R^n$ be a bounded domain with Lipschitz boundary. There exists a constant $C=C(\Omega,n,d)>0$ such that for all $u\in W^{1,1}(\Omega;\R^d)$ and for all $\lambda>0$ there exists a Lipschitz function $v\colon \Omega\to \R^d$ satisfying
\begin{itemize}
\item [(i)] $\|\nabla v\|_{L^\infty(\Omega)} \le C\lambda$,
\item [(ii)] $\{x\in \Omega\,:\,u(x)\neq v(x)\}\subseteq \{x\in \Omega\,:\,M_\Omega(\nabla u)(x)>\lambda\}$ (up to sets of measure zero).
\item [(iii)] $\displaystyle|\{x\in \Omega\,:\,u(x)\neq v(x)\}|\le C\int_{\{x\in\Omega\,:\,|\nabla u(x)|>\lambda\}}\frac{|\nabla u(z)|}{\lambda}\,\de z.$
\end{itemize}
\end{lemma}

\begin{proof}
The proof is analogous to the one of~\cite[Proposition~A.1]{FJM}, see also~\cite[Sections 6.6.2 and 6.6.3]{EG}. It is enought to repeat the argument used for cubes in Step 1 of~\cite[Proposition~A.1]{FJM} directly for Lipschitz domains by using~\eqref{eq:areaball} and the fact that we can find $r_0=r_0(\Omega)>0$ such that the Poincaré inequality holds in $\Omega\cap B_\rho(x)$ with a constant independent of $x\in \Omega$ and $\rho\in(0,r_0]$, see for example~\cite[Lemma~8.2.13]{DHHR}. 
\end{proof}

\begin{remark}\label{rem:on condition (ii)}
We point out that the Lusin approximation used in~\cite{FJM} states only (i) and (iii). In our paper instead, we need to explicitly point out the inclusion (ii), as condition (iii) is not enough to deduce the rigidity result of Theorem~\ref{thm:pxrig}, and later of Theorem~\ref{thm:pxqxrig}. This is basically due to the fact that inequality~\eqref{eq:pxrig} is stated in terms of the $L^{p(\cdot)}$-norm, while condition (iii) is given in terms of integrals. As underlined in~\cite{DHHR} for other functional inequalities, in a variable exponent setting one cannot in general deduce sharp inequalities between norms from those for the integrals, and vice versa. Inclusion (ii), instead, combined with continuity of the maximal operator, is enough to recover the crucial estimate~\eqref{eq:nablauv_est2} below. 
We keep however condition (iii) in the statement, as it is useful for the derivation of a rigidity estimate (Theorem~\ref{thm:gpxrig}) which is a main ingredient for the $\Gamma$-convergence result of Section~\ref{sec:4}, see also Remark~\ref{rem:norm-mod}. For general Lipschitz truncation results in Sobolev spaces with variable exponent, whose proof employs a similar condition to (ii), we refer the reader to~\cite{Diening:LipTruncation}. 
\end{remark}

 We start with the local rigidity estimate on a cube $Q$. We first prove the result for Sobolev functions with uniformly bounded gradients, and then we use conditions (i) and (ii) of the Lusin approximation to extend it to all Sobolev functions in $W^{1,p(\cdot)}(Q;\R^n)$. 

\begin{proposition}\label{prop:Q0pxrig}
 Let $r_0>0$ and let $Q_0:=(-r_0,r_0)^n$ and $Q_0':=(-\frac{r_0}{2},\frac{r_0}{2})^n$. Let $p\in\mathcal P^{\log}_b(Q_0)$ with $p^->1$. There exists a constant $C(r_0,n,p)>0$ such that for all $u\in W^{1,p(\cdot)}(Q_0;\R^n)$ we can find a constant rotation $R\in SO(n)$ satisfying
\begin{equation}\label{eq:Q0pxrig}
\|\nabla u-R\|_{L^{p(\cdot)}(Q_0')}\le C\|d(\nabla u,SO(n))\|_{L^{p(\cdot)}(Q_0)}.
\end{equation}
\end{proposition}

\begin{proof}
In view of Proposition~\ref{prop:417}, without loss of generality we may assume that $p\in \mathcal P^{\log}_b(\R^n)$ with $p^->1$. We start by proving that for every $M>0$ and functions $v\in W^{1,\infty}(Q_0;\R^n)$ with $\|\nabla v\|_{L^\infty(Q_0)}\le M$ there exists a constant $C=C(M,r_0,n,p)>0$ and a rotation $R\in SO(n)$ satisfying
\begin{equation}\label{eq:MQpxrig}
\|\nabla v-R\|_{L^{p(\cdot)}(Q_0')}\le C\|d(\nabla v,SO(n))\|_{L^{p(\cdot)}(Q_0)}.
\end{equation}

The function $A\mapsto |\cof A-A|^2$ is smooth and non negative on $\R^{n\times n}$ and it vanishes on $SO(n)$. Hence, there exists a constant $C=C(M,n)>0$ such that 
\begin{equation}\label{eq:cofA_est}
|\cof A-A|\le C d(A,SO(n))\quad\text{for all $A\in\R^{n\times n}$ with $|A|\le M$}.
\end{equation}
Since $\div \cof\nabla v=0$ in $\mathcal D'(Q_0;\R^n)$, the function $v$ satisfies
$$-\Delta v=\div(\cof\nabla v-\nabla v)\quad\text{in $\mathcal D'(Q_0;\R^n)$}.$$
We define $F\colon B_{2\sqrt{n}r_0}(0)\to \R^{n\times n}$ as
\begin{equation*}
F(x):=
\begin{cases}
\cof\nabla v(x)-\nabla v(x)&\text{if $x\in Q_0$},\\
0&\text{if $x\in B_{2\sqrt{n}r_0}(0)\setminus Q_0$}.
\end{cases}
\end{equation*}
 Since $F\in L^\infty(B_{2\sqrt{n}r_0}(0);\R^{n\times n})\subset L^{p(\cdot)}(B_{2\sqrt{n}r_0}(0);\R^{n\times n})$, by Proposition~\ref{prop:pxpoisson} there exists a unique strong solution $\Psi\in W^{2,p(\cdot)}(B_{2\sqrt{n}r_0}(0),\R^{n\times n})$ to 
$$\begin{cases}
-\Delta \Psi=F&\text{in $B_{2\sqrt{n}r_0}(0)$},\\
\Psi=0&\text{on $\partial B_{2\sqrt{n}r_0}(0)$}.
\end{cases}$$
In particular, the function $\psi:=\div\Psi$ lies in $ W^{1,p(\cdot)}(Q_0;\R^n)$ and by~\eqref{eq:ell-est} and~\eqref{eq:cofA_est} we have
\begin{equation}\label{eq:psi_est}
\begin{aligned}
\|\nabla \psi\|_{L^{p(\cdot)}(Q_0)}&\le \|\nabla^2\Psi\|_{L^{p(\cdot)}(B_{2\sqrt{n}r_0}(0))}\le C\|F\|_{L^{p(\cdot)}(B_{2\sqrt{n}r_0}(0))}\\
&=C\|\cof\nabla v-\nabla v\|_{L^{p(\cdot)}(Q_0)}\le C\|d(\nabla v,SO(n))\|_{L^{p(\cdot)}(Q_0)}
\end{aligned}
\end{equation}
for a constant $C= C(M,r_0,n,p)>0$. The function $\varphi:=v-\psi$ is harmonic in $Q_0$ since
$$-\Delta \varphi=-\Delta v+\Delta \psi=-\Delta v+\div\Delta \Psi=0\quad\text{in $\mathcal D'(Q_0;\R^n)$}.$$
Hence, it can be represented by a $C^\infty$ function in $Q_0$ by Weyl's lemma. 

 Since $\varphi\in W^{1,p(\cdot)}(Q_0;\R^n)\subseteq W^{1,p^-}(Q_0;\R^n)$, by the rigidity result in $W^{1,p^-}(Q_0;\R^n)$ (see~\cite[Section 2.4]{conti.schweizer}) and~\eqref{eq:psi_est} there exists a rotation $R\in SO(n)$ satisfying
\begin{equation}
\begin{aligned}
\|\nabla \varphi-R\|_{L^{p^-}(Q_0)}&\le C\|d(\nabla\varphi,SO(n))\|_{L^{p^-}(Q_0)}\le C\|d(\nabla\varphi,SO(n))\|_{L^{p(\cdot)}(Q_0)}\\
&\le C(\|d(\nabla v,SO(n))\|_{L^{p(\cdot)}(Q_0)}+\|\nabla \psi\|_{L^{p(\cdot)}(Q_0)})\\
&\le C\|d(\nabla v,SO(n))\|_{L^{p(\cdot)}(Q_0)}
\end{aligned}
\end{equation}
for a constant $C=C(M,r_0,n,p)>0$. Thanks to the fact that $\nabla\varphi-R$ is harmonic in $Q_0$, we can use the mean value property to derive
\begin{equation}\label{eq:mean-harm}
\|\nabla \varphi-R\|_{L^{p(\cdot)}(Q_0')}\le C\|\nabla \varphi-R\|_{L^\infty(Q_0')}\le C\|\nabla \varphi-R\|_{L^1(Q_0)}\le C\|\nabla \varphi-R\|_{L^{p(\cdot)}(Q_0)} 
\end{equation}
for a constant $C=C(M,r_0,n,p)>0$. By combining~\eqref{eq:psi_est}--\eqref{eq:mean-harm} we deduce~\eqref{eq:MQpxrig}.

Finally, let $u\in W^{1,p(\cdot)}(Q_0;\R^n)$ and let $v\in W^{1,\infty}(Q_0;\R^n)$ be the function given by Lemma~\ref{lem:lusin} with $\lambda=2\sqrt{n}$. Then, there exists a constant $C=C(Q_0,n)>0$ such that
$$\|\nabla v\|_{L^\infty(Q_0)}\le 2\sqrt{n}C,\quad \{x\in Q_0\,:\, u(x)\neq v(x)\}\overset{a.e.}{\subseteq}\{x\in Q_0\,:\, M_{Q_0}(\nabla u)(x)>2\sqrt{n}\}.$$
Since 
\begin{equation}\label{eq:est2}
|\nabla u|\le d(\nabla u,SO(n))+\sqrt{n},\quad M_{Q_0}(\nabla u)\le M_{Q_0}(d(\nabla u,SO(n)))+\sqrt{n},
\end{equation}
we derive 
\begin{align}\label{eq:set_est}
\{x\in Q_0\,:\,M_{Q_0}(\nabla u)(x)>2\sqrt{n}\}\subseteq\{x\in Q_0\,:\,M_{Q_0}(d(\nabla u, SO(n)))(x)>\sqrt{n}\},
\end{align}
which gives
\begin{align*}
&\|\nabla u-\nabla v\|_{L^{p(\cdot)}(Q_0)}\\
&=\|\nabla u-\nabla v\|_{L^{p(\cdot)}(\{x\in Q_0\,:\,u(x)\neq v(x)\})}\\
&\le \|\nabla u\|_{L^{p(\cdot)}(\{x\in Q_0\,:\,|\nabla u(x)|>2\sqrt{n}\})}+2\sqrt{n}(C+1)\|1\|_{L^{p(\cdot)}(\{x\in Q_0\,:\,u(x)\neq v(x)\})}\\
&\le \|\nabla u\|_{L^{p(\cdot)}(\{x\in Q_0\,:\,|\nabla u(x)|>2\sqrt{n}\})}+2(C+1)\|\sqrt{n}\|_{L^{p(\cdot)}(\{x\in Q_0\,:\,M_{Q_0}(d(\nabla u, SO(n)))(x)>\sqrt{n}\})}.
\end{align*}
Therefore, by Proposition~\ref{prop:maxfun2} and~\eqref{eq:est2} we deduce
\begin{equation}\label{eq:nablauv_est2}
\begin{aligned}
\|\nabla u-\nabla v\|_{L^{p(\cdot)}(Q_0)}
&\le 2\|d(\nabla u,SO(n))\|_{L^{p(\cdot)}(Q_0)}+2(C+1)\|M_{Q_0}(d(\nabla u,SO(n)))\|_{L^{p(\cdot)}(Q_0)}\\
&\le C\|d(\nabla u,SO(n))\|_{L^{p(\cdot)}(Q_0)}
\end{aligned}
\end{equation}
for a constant $C=C(r_0,n)>0$. By~\eqref{eq:MQpxrig} there exists a constant $C=C(r_0,n,p)>0$ and a rotation $R\in SO(n)$ such that 
\begin{equation*}
\|\nabla v-R\|_{L^{p(\cdot)}(Q_0')}\le C\|d(\nabla v,SO(n))\|_{L^{p(\cdot)}(Q_0)}.
\end{equation*}
Thus, we have
\begin{align*}
\|\nabla u-R\|_{L^{p(\cdot)}(Q_0')}&\le\|\nabla u-\nabla v\|_{L^{p(\cdot)}(Q_0')}+\|\nabla v-R\|_{L^{p(\cdot)}(Q_0')}\\
&\le \|\nabla u-\nabla v\|_{L^{p(\cdot)}(Q_0')}+C\|d(\nabla v,SO(n))\|_{L^{p(\cdot)}(Q_0)}\\
&\le C(\|\nabla u-\nabla v\|_{L^{p(\cdot)}(Q_0)}+C\|d(\nabla u,SO(n))\|_{L^{p(\cdot)}(Q_0)})\\
&\le C\|d(\nabla u,SO(n))\|_{L^{p(\cdot)}(Q_0)}
\end{align*}
for a constant $C=C(r_0,n,p)>0$. This gives~\eqref{eq:Q0pxrig} and concludes the proof. 
\end{proof}

 Next, we show that the rigidity estimates~\eqref{eq:Q0pxrig} holds in every cube $Q=a+(-r,r)^n$ with a rigidity constant $C$ independent of $a$ and $r$, provided that $r$ is uniformly bounded. This can be done by means of the $\log$-Hölder condition of $p$. 

\begin{proposition}\label{prop:Qpxrig}
Let $p\in\mathcal P^{\log}_b(\R^n)$ be such that $p^->1$. Let $r_0>0$ be fixed. There exists a constant $C=C(r_0,n,p)>0$ such that for all cubes $Q:=a+(-r,r)^n$ and $Q':=a+(-\frac{r}{2},\frac{r}{2})^n$ with $0<r\le r_0$, and for all functions $u\in W^{1,p(\cdot)}(Q;\R^n)$ we can find a constant rotation $R\in SO(n)$ satisfying
\begin{equation}\label{eq:Qpxrig}
\|\nabla u-R\|_{L^{p(\cdot)}(Q')}\le C\|d(\nabla u,SO(n))\|_{L^{p(\cdot)}(Q)}. 
\end{equation}
\end{proposition}

\begin{proof}
We set $Q_0:=(-r_0,r_0)^n$ and $Q_0':=(-\frac{r_0}{2},\frac{r_0}{2})^n$. For all $0<r\le r_0$, we consider the functions 
$\psi_{a,\frac{r}{r_0}}(x)=a+\frac{r}{r_0}x\quad\text{for $x\in \R^n$}$ 
introduced in~\eqref{eq:psilambda}. Let $u\in W^{1,p(\cdot)}(Q;\R^n)$ and define $v:=\frac{r_0}{r}u\circ\psi_{a,\frac{r}{r_0}}$ and $q:=p\circ\psi_{a,\frac{r}{r_0}}\in\mathcal P_b(Q_0)$. We have that $v\in W^{1,q(\cdot)}(Q_0;\R^n)$ by Proposition~\ref{prop:change} and
$$\nabla v(x)=\nabla u(\psi_{a,\frac{r}{r_0}}(x))\quad\text{for a.e.\ $x\in Q_0$}.$$
Notice that $q$ satisfies
\begin{align}\label{eq:MMM1}
p^-\le q^-_{Q_0}\le q^+_{Q_0}\le p^+,
\end{align}
and for all $x_1,x_2, x \in Q_0$ we have by~\eqref{eq:logcon}--\eqref{eq:logcon2} and being $0<r\le r_0$ 
\begin{align}
&|q(x_1)-q(x_2)|\le \frac{c_{\log}(p)}{\log(\e+r_0/(r|x_1-x_2|))}\le \frac{c_{\log}(p)}{\log(\e+1/|x_1-x_2|))},\label{eq:MMM2}\\
& |q(x)| \le \frac{p^+\log(\e+ \sqrt{n} r_0 )}{\log(\e +|x|)}. \notag
\end{align}
Hence, $q\in\mathcal P_b^{\log}(Q_0)$ and the $\log$-Hölder constant of $q$ on $Q_0$ is uniformly bounded by a constant depending only on $n$, $r_0$, $p^+$, and $c_{\log}(p)$. By Proposition~\ref{prop:Q0pxrig} there exists a constant $C_1= C_1(r_0,n,q) >0$ such that we can find a constant rotation $R\in SO(n)$ (depending on $v$ and thus on $u$, $a$, and $r$) satisfying
$$\|\nabla v-R\|_{L^{q(\cdot)}(Q_0')}\le C_1\|d(\nabla v,SO(n))\|_{L^{q(\cdot)}(Q_0)}.$$
By Remark~\ref{rem:p} we note that $C_1(r_0,n,q) = C_1(r_0,n,q^-,q^+,c_{\log}(q))$ and then, in view of~\eqref{eq:MMM1}--\eqref{eq:MMM2}, we get $C_1(r_0,n,q) = C_1(r_0,n,p)$. Hence,
$$\|\nabla u\circ\psi_{a,\frac{r}{r_0}}-R\|_{L^{q(\cdot)}(Q_0')}\le C_1\|d(\nabla u\circ\psi_{a,\frac{r}{r_0}},SO(n))\|_{L^{q(\cdot)}(Q_0)},$$
and by Proposition~\ref{prop:change} for $f=\nabla u-R$ and $f=d(\nabla u,SO(n))$ we obtain 
\begin{align*}
\|\nabla u-R\|_{L^{p(\cdot)}(Q')}&\le C_1 \left(\frac{r}{r_0}\right)^{\frac{n}{p_Q^+}}\left(\frac{r}{r_0}\right)^{-\frac{n}{p_Q^-}} \|d(\nabla u,SO(n))\|_{L^{p(\cdot)}(Q)}. 
\end{align*}
In view of $p\in\mathcal P^{\log}_b(\R^n)$ and Proposition~\ref{prop:smallQ}, the right-hand side is controlled by 
\begin{align*}
r_0^{\frac{n(p_{Q}^+-p_{Q}^-)}{p^-_{Q}p^+_{Q}}}r^{\frac{n(p_{Q}^--p_{Q}^+)}{p^-_{Q}p^+_{Q}}}C_1 \|d(\nabla u,SO(n))\|_{L^{p(\cdot)}(Q)}\le (1+r_0^n)C_1C_2\|d(\nabla u,SO(n))\|_{L^{p(\cdot)}(Q)}
\end{align*}
for a constant $C_2=C_2(n,p)>0$. Therefore, inequality~\eqref{eq:Qpxrig} is satisfied for a constant $C=C(r_0,n,p)>0$. 
\end{proof}

We point out that the assumptions of Proposition~\ref{prop:Qpxrig} are easily satisfied if all cubes are contained in a bounded set. This allows us to extend the rigidity result of Proposition~\ref{prop:Q0pxrig} to the case of bounded domains with Lipschitz boundary. 

 In order to prove Theorem~\ref{thm:pxrig} we also need the following weighted Poincaré inequality, here specified for a variable exponent, whose proof is postponed to the Appendix. 

\begin{proposition}[Weighted Poincaré inequality]\label{prop:poincare}
Let $\Omega\subset\R^n$ be a bounded domain with Lipschitz boundary. Let $p\in\mathcal P^{\log}_b(\R^n)$ with $p^->1$. There exists a constant $C=C(\Omega,n,d,p)>0$ such that for every locally Lipschitz function $f\in L^{p(\cdot)}(\Omega;\R^d)$ we can find a constant vector $a\in \R^d$ satisfying
\begin{equation}\label{eq:poincare}
\|f-a\|_{L^{p(\cdot)}(\Omega)}\le C\|d(\cdot,\partial\Omega)\nabla f\|_{L^{p(\cdot)}(\Omega)}. 
\end{equation}
\end{proposition}

We can finally prove Theorem~\ref{thm:pxrig}. The proof, which relies on a Whitney covering of $\Omega$, follows the strategy adopted by~\cite{FJM}, in the version proposed by~\cite{conti.garroni}. 

\begin{proof}[Proof of Theorem~\ref{thm:pxrig}]
 We consider a Whitney
covering of $\Omega$, i.e., a countable family of cubes $\{Q_i\}_i$, where $Q_i:=a_i+(-\frac{r_i}{2},\frac{r_i}{2})^n\subset\Omega$ with $a_i\in\R^n$ and $r_i>0$, which satisfies: there exists a constant $C=C(n)>0$ such that 
\begin{equation}\label{eq:whitney}
\sqrt{n} r_i\le d(Q_i,\partial \Omega)\le Cr_i\quad\text{for all $i\in\N$}, 
\end{equation}
and there exists $N=N(\Omega,n)>0$ such that
\begin{equation}\label{eq:locfin}
\chi_\Omega\le\sum_{i=1}^\infty\chi_{Q_i}\le\sum_{i=1}^\infty\chi_{2Q_i}\le N\chi_\Omega\quad\text{in $\R^n$}, 
\end{equation}
where we set $2Q_i:=a_i+(-r_i,r_i)^n\subset\Omega$. Notice that the family $\{Q_i\}_i$ satisfies the assumptions of Proposition~\ref{prop:Qpxrig}. Hence, there exists a constant $C=C(\Omega,n,p)>0$, independent of $i\in\N$, and constant rotations $R_i\in SO(n)$ such that 
\begin{equation}\label{eq:Qi-rig}
\|\nabla u-R_i\|_{L^{p(\cdot)}(Q_i)}\le C\|d(\nabla u,SO(n))\|_{L^{p(\cdot)}(2Q_i)}.
\end{equation}
Moreover, by the $\log$-Hölder continuity of $p$, using Remark~\ref{rem:volume} and Proposition~\ref{prop:smallQ}, we observe that there exists a constant $C(n,p)>0$ such that 
\begin{equation}\label{eq:2Qi-Qi}
\|\chi_{2Q_i}\|_{L^{p(\cdot)}(\R^n)}\le C\|\chi_{Q_i}\|_{L^{p(\cdot)}(\R^n)}.
\end{equation}

Let $\{\phi_i\}_i$ be a partition of unity subordinated to the the covering $\{Q_i\}_i$, that is for all $i\in\N$
\begin{equation}\label{eq:phii}
\phi_i\in C_c^\infty(Q_i),\quad 0\le \phi_i\le 1\text{ in }Q_i,\quad \|\nabla\phi_i\|_{L^{\infty}(Q_i)}\le \frac{C}{r_i},\quad \sum_{i=1}^\infty\phi_i=1\text{ on }\Omega,
\end{equation}
with constant $C=C(n)>0$ independent of $i\in\N$. We define the function $\hat R\colon\Omega\to \R^{n\times n}$ as 
$$\hat R(x):=\sum_{i=1}^\infty R_i\phi_i(x)\quad\text{for all $x\in\Omega$}.$$
Since $\hat R\in C^\infty(\Omega;\R^{n\times n})\cap L^\infty(\Omega;\R^{n\times n})$, by Proposition~\ref{prop:poincare} there exists $A\in \R^{n\times n}$ such that 
\begin{equation}\label{eq:hatR-A}
\|\hat R-A\|_{L^{p(\cdot)}(\Omega)}\le C\|d(\cdot,\partial\Omega)\nabla \hat R\|_{L^{p(\cdot)}(\Omega)}.
\end{equation}
Moreover, by Propositions~\ref{prop:locglob} and~\ref{prop:locglob2} we have
\begin{equation}\label{eq:estimatebr}
\begin{aligned}
\|\nabla u-\hat R\|_{L^{p(\cdot)}(\Omega)}\underset{\eqref{eq:phii}}{\le}& \left\|\sum_{i=1}^\infty \chi_{Q_i}|\nabla u-R_i|\right\|_{L^{p(\cdot)}(\R^n)}\\
\underset{\eqref{eq:locglob2}}{\le}&C\left\|\sum_i\chi_{Q_i}\frac{\|\nabla u-R_i\|_{L^{p(\cdot)}(Q_i)}}{\|\chi_{Q_i}\|_{L^{p(\cdot)}(\R^n)}}\right\|_{L^{p(\cdot)}(\R^n)}\\
\underset{\eqref{eq:Qi-rig}}{\le} &C\left\|\sum_i\chi_{Q_i}\frac{\|d(\nabla u,SO(n))\|_{L^{p(\cdot)}(2Q_i)}}{\|\chi_{Q_i}\|_{L^{p(\cdot)}(\R^n)}}\right\|_{L^{p(\cdot)}(\R^n)}\\
\underset{\eqref{eq:2Qi-Qi}}{\le} &C\left\|\sum_i\chi_{2Q_i}\frac{\|d(\nabla u,SO(n))\|_{L^{p(\cdot)}(2Q_i)}}{\|\chi_{2Q_i}\|_{L^{p(\cdot)}(\R^n)}}\right\|_{L^{p(\cdot)}(\R^n)}\\
\underset{\eqref{eq:locglob}}{\le} &C\left\|\sum_i\chi_{2Q_i}d(\nabla u,SO(n))\right\|_{L^{p(\cdot)}(\R^n)}\\
\underset{\eqref{eq:locfin}}{\le} &\|d(\nabla u,SO(n))\|_{L^{p(\cdot)}(\Omega)}
\end{aligned}
\end{equation}
for a constant $C=C(\Omega,n,p)>0$. Moreover, we have
\begin{equation}\label{eq:nablaR}
\nabla \hat R(x)=\sum_{i=1}^\infty R_i\nabla\phi_i(x)=\sum_{i=1}^\infty (R_i-\nabla u(x))\nabla\phi_i(x)\quad\text{for a.e.\ $x\in\Omega$},
\end{equation}
since $\sum_{i=1}^\infty \nabla\phi_i=0$ in $\Omega$, which gives
\begin{equation}\label{eq:dist-hatR}
\begin{aligned}
\|d(\cdot,\partial\Omega)\nabla \hat R\|_{L^{p(\cdot)}(\Omega)}\underset{\eqref{eq:nablaR}}{\le}& \left\|\sum_{i=1}^\infty \chi_{Q_i}d(\cdot,\partial\Omega)|\nabla \phi_i||\nabla u-R_i|\right\|_{L^{p(\cdot)}(\R^n)}\\
\underset{\eqref{eq:whitney},\,\eqref{eq:phii}}{\le}& C \left\|\sum_{i=1}^\infty \chi_{Q_i}|\nabla u-R_i|\right\|_{L^{p(\cdot)}(\R^n)}\\
\underset{\eqref{eq:estimatebr}}{\le} &C\|d(\nabla u,SO(n))\|_{L^{p(\cdot)}(\Omega)}
\end{aligned}
\end{equation}
for a constant $C=C(\Omega,n,p)>0$. Let $R\in SO(n)$ be such that $|R-A|=d(A,SO(n))$. Then, by~\eqref{eq:hatR-A},~\eqref{eq:estimatebr}, and~\eqref{eq:dist-hatR} we obtain
\begin{align*}
\|\nabla u-R\|_{L^{p(\cdot)}(\Omega)}&\le \|\nabla u-\hat R\|_{L^{p(\cdot)}(\Omega)}+\|\hat R-A\|_{L^{p(\cdot)}(\Omega)}+\|d(A,SO(n))\|_{L^{p(\cdot)}(\Omega)} \\
&\le 2\|\nabla u-\hat R\|_{L^{p(\cdot)}(\Omega)}+2\|\hat R-A\|_{L^{p(\cdot)}(\Omega)}+\|d(\nabla u,SO(n))\|_{L^{p(\cdot)}(\Omega)}\\
&\le C\|d(\nabla u,SO(n))\|_{L^{p(\cdot)}(\Omega)},
\end{align*}
 which gives~\eqref{eq:pxrig}.

Finally, if $R\in SO(n)$ satisfies~\eqref{eq:pxrig} and $\hat R\in SO(n)$ satisfies~\eqref{eq:R-mean}, then
\begin{align*}
\|\nabla u-\hat R\|_{L^{p(\cdot)}(\Omega)}&\le\|\nabla u-R\|_{L^{p(\cdot)}(\Omega)}+\|R-\langle\nabla u\rangle_\Omega\|_{L^{p(\cdot)}(\Omega)}+\|d(\langle\nabla u\rangle_\Omega,SO(n))\|_{L^{p(\cdot)}(\Omega)}\\
&\le 2\|\nabla u-R\|_{L^{p(\cdot)}(\Omega)}+2\|R-\langle\nabla u\rangle_\Omega\|_{L^{p(\cdot)}(\Omega)}+\|d(\nabla u,SO(n))\|_{L^{p(\cdot)}(\Omega)}\\
&\le C\|\nabla u-R\|_{L^{p(\cdot)}(\Omega)}+\|d(\nabla u,SO(n))\|_{L^{p(\cdot)}(\Omega)}\le C\|d(\nabla u,SO(n))\|_{L^{p(\cdot)}(\Omega)}
\end{align*}
for a constant $C=C(\Omega,n,p)>0$. Therefore, in~\eqref{eq:pxrig} we can take $R$ satisfying~\eqref{eq:R-mean}. In particular, the constant rotation $R\in SO(n)$ depends only on $u$ and $\Omega$.
\end{proof}

\section{Geometric rigidity with mixed growth conditions and variable exponents}\label{sec:3}

This section is devoted to the generalization of the geometric rigidity with mixed growth conditions of~\cite[Theorem 1.1]{CDM} to the case of variable exponents. As a preparation, we need to generalize Korn's inequality with mixed growth conditions (see~\cite[Theorem~2.1]{CDM}) to the case of variable exponents, see Theorem~\ref{thm:pxqxkorn} below. The main difference to the proof presented in~\cite{CDM} is that we need to use a suitable extension result in the variable exponents setting.

 Let us start by formulating this extension result. For the proof, we refer to the Appendix. It extends to the variable exponent setting a technique devised in~\cite[Theorem 5.1]{CDM}, in its turn inspired by~\cite[Lemma 4]{nitsche}. 
\begin{theorem}[Extension result]\label{thm:extension}
Let $\phi\in\Lip(\R^{n-1};\R)$ be a Lipschitz function with $\phi(0)=0$ and Lipschitz constant $L>0$. Let us define
$$\Omega:=\{x=(x',x_n)\in\R^{n-1}\times \R\,:\, x_n<\phi(x')\}.$$
Let $R>0$ and let $p,q\in\mathcal P^{\log}_b(B_R(0)\cap\Omega)$ be such that
$$p(x)\le q(x)\quad\text{for a.e.\ $x\in B_R(0)\cap\Omega$}.$$
Assume that $u\in W^{1,1}(B_R(0)\cap\Omega;\R^n)$, $f\in L^{p(\cdot)}(B_R(0)\cap\Omega;\R^{n\times n})$, and $g\in L^{q(\cdot)}(B_R(0)\cap\Omega;\R^{n\times n})$ satisfy
$$ eu(x)=f(x)+g(x)\quad\text{for a.e.\ $x\in B_R(0)\cap\Omega$}.$$
There exist a radius $r=r(R,L,n)$, with $0<r<R$, a constant $C=C(R,L,n,p,q)>0$ and functions $\tilde p,\tilde q\in \mathcal P^{\log}_b(B_r(0))$, $\tilde u\in W^{1,1}(B_r(0);\R^n)$, $\tilde f\in L^{\tilde p(\cdot)}(B_r(0);\R^{n\times n})$, and $\tilde g\in L^{\tilde q(\cdot)}(B_r(0);\R^{n\times n})$ satisfying
\begin{align*}
&\tilde p(x)=p(x),\quad\tilde q(x)=q(x)& &\text{for a.e $x\in B_r(0)\cap \Omega$},\\
&p^-\le \tilde p^-\le \tilde p^+\le p^+,\quad q^-\le \tilde q^-\le \tilde q^+\le q^+,\quad\tilde p(x)\le \tilde q(x)& &\text{for a.e.\ $x\in B_r(0)$},\\
&\tilde u(x)=u(x),\quad \tilde f(x)=f(x),\quad\tilde g(x)=g(x)& &\text{for a.e $x\in B_r(0)\cap \Omega$},\\
&e\tilde u(x)=\tilde f(x)+\tilde g(x)& &\text{for a.e.\ $x\in B_r(0)$},
\end{align*}
and such that
\begin{equation}\label{eq:fgest}
\|\tilde f\|_{L^{\tilde p(\cdot)}(B_r(0))}\le C\|f\|_{L^{p(\cdot)}(B_R(0)\cap\Omega)},\quad \|\tilde g\|_{L^{\tilde q(\cdot)}(B_r(0))}\le C\|g\|_{L^{ q(\cdot)}(B_R(0)\cap\Omega)}. 
\end{equation}
\end{theorem}

The proof of Theorem~\ref{thm:extension} is postponed to the Appendix. We can now state the Korn inequality with mixed growth conditions and variable exponents.

\begin{theorem}[Korn's inequality for mixed growth and variable exponents]\label{thm:pxqxkorn}
Let $\Omega\subset\R^n$ be a bounded domain with Lipschitz boundary. Let $p,q\in
\mathcal P^{\log}_b(\Omega)$ be such that
$$ p^->1, \quad p(x)\le q(x)\quad\text{for a.e.\ $x\in\Omega$}.$$
Assume that $u\in W^{1,1}(\Omega;\R^n)$, $f\in L^{p(\cdot)}(\Omega;\R^{n\times n})$, and $g\in L^{q(\cdot)}(\Omega;\R^{n\times n})$ satisfy
$$ eu(x)=f(x)+g(x)\quad\text{for a.e.\ $x\in\Omega$}.$$
There exist a constant $C=C(\Omega,n,p,q)>0$, a skew-symmetric matrix $S\in \R^{n\times n}_{\skw}$, and two functions $F\in L^{p(\cdot)}(\Omega;\R^{n\times n})$ and $G\in L^{q(\cdot)}(\Omega;\R^{n\times n})$ satisfying
\begin{equation}\label{eq:Spqkorn}
\nabla u(x)-S=F(x)+G(x)\quad\text{for a.e.\ $x \in\Omega$},
\end{equation}
and
\begin{equation}\label{eq:Spq_est}
\|F\|_{L^{p(\cdot)}(\Omega)}\le C \|f\|_{L^{p(\cdot)}(\Omega)}\quad\text{and}\quad \|G\|_{L^{q(\cdot)}(\Omega)}\le C \|g\|_{L^{q(\cdot)}(\Omega)}.
\end{equation}
In particular, we can take $S=(\langle\nabla u\rangle_\Omega)_{\skw}$.
\end{theorem}

\begin{proof}
Suppose first that
$$\|g\|_{L^{q(\cdot)}(\Omega)}\le \|f\|_{L^{p(\cdot)}(\Omega)}.$$
By the Korn inequality in Proposition~\ref{prop:pxkorn} there exist a constant $C=C(\Omega,n,p)>0$ and a constant skew-symmetric matrix $S\in \R^{n\times n}_{\skw}$ such that 
$$ \|\nabla u- S\|_{L^{p(\cdot)}(\Omega)}\le C\|eu\|_{L^{p(\cdot)}(\Omega)},$$
where $S$ can be taken as $(\langle\nabla u\rangle_\Omega)_{\skw}$. Then, Theorem~\ref{thm:pxqxkorn} is true by taking $F=\nabla u-S$ and $G=0$. Indeed, we have by Proposition~\ref{prop:spq} 
\begin{align*}
 \|\nabla u- S\|_{L^{p(\cdot)}(\Omega)}&\le C \|eu\|_{L^{p(\cdot)}(\Omega)}\le C\left(\|f\|_{L^{p(\cdot)}(\Omega)}+\|g\|_{L^{p(\cdot)}(\Omega)}\right)\\
&\le C\left(\|f\|_{L^{p(\cdot)}(\Omega)}+\|g\|_{L^{q(\cdot)}(\Omega)}\right)\le C\|f\|_{L^{p(\cdot)}(\Omega)}
\end{align*}
for a constant $C=C(\Omega,n,p,q)>0$. Therefore, from now on we assume that 
\begin{align}\label{fromnowon}
\|g\|_{L^{q(\cdot)}(\Omega)}>\|f\|_{L^{p(\cdot)}(\Omega)}.
 \end{align}
{\bf Step 1.} Let $x\in\Omega$ and let $r>0$ such that $B_r(x)\subset\Omega$. The function $u$ satisfies
\begin{align}\label{deltau}
-\Delta u=\nabla (\tr eu )-2\div eu=\div((\tr eu) I-2eu)\quad\text{in $\mathcal D'(B_r(x);\R^n)$}.
\end{align}
We consider the unique strong solutions $\Psi_f\in W^{2,p(\cdot)}(B_r(x);\R^{n\times n})$ and $\Psi_g\in W^{2,q(\cdot)}(B_r(x);\R^{n\times n})$ to the problems
$$ \begin{cases}
-\Delta \Psi_f=(\tr f) I-2f&\text{in $B_r(x)$},\\
\Psi_f=0&\text{on $\partial B_r(x)$},
\end{cases}\quad
\begin{cases}
-\Delta \Psi_g=(\tr g) I-2g&\text{in $B_r(x)$},\\
\Psi_g=0&\text{on $\partial B_r(x)$},
\end{cases}
$$
given by Proposition~\ref{prop:pxpoisson}, and we define $u_f:=\div \Psi_f$ and $u_g:=\div\Psi_g$. By Proposition~\ref{prop:pxpoisson} we can find a constant $C=C(r,x,n,p,q)>0$ such that 
\begin{align}\label{ufug1}
&\|\nabla u_f\|_{L^{p(\cdot)}(B_r(x))}\le \|\nabla^2 \Psi_f\|_{L^{p(\cdot)}(B_r(x))}\le C\|f\|_{L^{p(\cdot)}(B_r(x))},
\end{align}
\begin{align}\label{ufug2}
&\|\nabla u_g\|_{L^{q(\cdot)}(B_r(x))}\le \|\nabla^2 \Psi_g\|_{L^{q(\cdot)}(B_r(x))}\le C\|g\|_{L^{q(\cdot)}(B_r(x))}. 
\end{align}
 In view of~\eqref{deltau} and the definition of $u_f$, $u_g$, the function $w:=u-u_f-u_g\in W^{1,1}(B_r(x);\R^n)$ satisfies $\Delta w = 0$ in $\mathcal D'(B_r(x);\R^n)$. Hence, by Weyl's lemma $w$ can be identified with a smooth harmonic function on $B_r(x)$. Since the function $ew$ is also harmonic, by the mean value property of harmonic functions, we have
$$\|ew\|_{L^{q^+}(B_{\frac{r}{2}}(x))}\le C\|ew\|_{L^{p^-}(B_r(x))}$$
for a constant $C= C(r,n,p,q) >0$. Thus, by using H\"older's inequality in Proposition~\ref{prop:spq} multiple times and employing~\eqref{ufug1}--\eqref{ufug2}, we deduce
\begin{align*}
\|ew\|_{L^{q(\cdot)}(B_{\frac{r}{2}}(x))}&\le C\|ew\|_{L^{q^+}(B_{\frac{r}{2}}(x))}\le C\|ew\|_{L^{p^-}(B_r(x))}\le C\|ew\|_{L^{p(\cdot)}(B_r(x))}\\
&\le C\|eu\|_{L^{p(\cdot)}(B_r(x))}+C\|\nabla u_f\|_{L^{p(\cdot)}(B_r(x))}+C\|\nabla u_g\|_{L^{p(\cdot)}(B_r(x))}\\
&\le C\left(\|f\|_{L^{p(\cdot)}(B_r(x))}+\|g\|_{L^{q(\cdot)}(B_r(x))}\right)
\end{align*}
for a constant $C=C(x,r,n,p,q)>0$. By Korn's inequality in $B_{\frac{r}{2}}(x)$, see Proposition~\ref{prop:pxkorn}, there exist a constant $C=C(r,x,n,q)>0$ and a skew-symmetric matrix $S\in \R^{n\times n}_{\skw}$ such that
$$\|\nabla w-S\|_{L^{q(\cdot)}(B_{\frac{r}{2}}(x))}\le C\|ew\|_{L^{q(\cdot)}(B_{\frac{r}{2}}(x))}.$$
By setting $F=\nabla u_f$ and $G=\nabla u-\nabla u_f-S=\nabla w+\nabla u_g-S$ we get
$$ \nabla u-S=F+G\quad\text{a.e.\ in $B_r(x)$}$$
and, again by~\eqref{ufug1}--\eqref{ufug2}, 
\begin{align}\label{F,G}
\|F\|_{L^{p(\cdot)}(B_{\frac{r}{2}}(x))}\le C\|f\|_{L^{p(\cdot)}(B_r(x))},\quad \|G\|_{L^{q(\cdot)}(B_{\frac{r}{2}}(x))}\le C\left(\|f\|_{L^{p(\cdot)}(B_r(x))}+\|g\|_{L^{q(\cdot)}(B_r(x))}\right).
\end{align}

{\bf Step 2.} For all $x\in\Omega$ we consider a radius $R_x>0$ such that $B_{R_x}(x)\subset\Omega$ and for all $x\in\partial\Omega$ we consider a radius $R_x>0$ with the following properties: there exists orthonormal vectors $v_1,\dots,v_n\in\R^n$ which determine a coordinate system in $\R^n$ and a Lipschitz function $\phi_x\colon\R^{n-1}\to\R$ such that $\phi_x(0)=0$ and
\begin{align}\label{eq:Lipboundary}
B_{R_x}(x)\cap \Omega=\left\{x+\sum_{i=1}^n\xi_iv_i\in\R^n\,:\, \xi \in B_{R_x}(0),\,\xi_n<\phi_x(\xi_1,\dots,\xi_{n-1})\right\}.
\end{align}
Let $L>0$ be the uniform Lipschitz constant of all the functions $\phi_x$. For every $x\in\partial\Omega$ we define $r_x$ as the radius associated to $R_x$ given by Theorem~\ref{thm:extension}, while for $x\in\Omega$ we take $r_x=R_x$. By construction, the family $\{B_{\frac{r_x}{2}}(x)\}_{x\in\overline\Omega}$ is a cover of $\overline\Omega$, which is compact. Then there exists a finite subcover $\{B_{\frac{r_i}{2}}(x_i)\}_{i\in\{1,\dots,N\}}$ with $N=N(\Omega)\in\N$. 

If $i\in\{1,\dots,N\}$ is such that $B_{r_i}(x_i)\subset\Omega$, then by Step 1 there exist $C_i=C_i(n,p,q)>0$, $S_i\in \R^{n\times n}_{\skw}$, $F_i\in L^{p(\cdot)}(B_{r_i}(x_i);\R^{n\times n})$, and $G\in L^{q(\cdot)}(B_{r_i}(x_i);\R^{n\times n})$ satisfying
$$\nabla u-S_i=F_i+G_i\quad\text{a.e.\ in $B_{r_i}(x_i)$},$$
and, using~\eqref{fromnowon} and~\eqref{F,G}, 
$$\|F_i\|_{L^{p(\cdot)}(B_{\frac{r_i}{2}}(x_i))}\le C_i \|f\|_{L^{p(\cdot)}(\Omega)}, \quad \|G_i\|_{L^{q(\cdot)}(B_{\frac{r_i}{2}}(x_i))}\le C_i \|g\|_{L^{q(\cdot)}(\Omega)}.$$
Otherwise, $i\in\{1,\dots,N\}$ is such that $ B_{R_i}(x_i)\cap \Omega$ satisfies~\eqref{eq:Lipboundary} for a function $\phi_i\in\Lip(\R^{n-1};\R)$ with $\phi_i(0)=0$ and Lipschitz constant $L>0$. We apply Theorem~\ref{thm:extension} and we find $\tilde p_i,\tilde q_i\in\mathcal P^{\log}_b(B_{r_i}(x_i))$, $\tilde u_i\in W^{1,1}(B_{r_i}(x_i);\R^n)$, $\tilde f_i\in L^{\tilde p_i(\cdot)}(B_{r_i}(x_i);\R^{n\times n})$, and $\tilde g_i\in L^{\tilde q_i(\cdot)}(B_{r_i}(x_i);\R^{n\times n})$ satisfying
$$ e\tilde u_i=\tilde f_i+\tilde g_i\quad\text{a.e.\ on $B_{r_i}(x_i)$}.$$
Hence by Step 1 there exist $C_i=C_i(\Omega,n,p,q)>0$, $S_i\in \R^{n\times n}_{\skw}$, $ F_i\in L^{\tilde p_i(\cdot)}(B_{r_i}(x_i);\R^{n\times n})$, and $ G_i\in L^{\tilde q_i(\cdot)}(B_{r_i}(x_i);\R^{n\times n})$ satisfying
$$\nabla \tilde u_i-S_i= F_i+ G_i\quad\text{a.e.\ in $B_{r_i}(x_i)$},$$
and employing~\eqref{eq:fgest},~\eqref{fromnowon}, and~\eqref{F,G}, 
\begin{align*}
&\| F_i\|_{L^{\tilde p_i(\cdot)}(B_{\frac{r_i}{2}}(x_i))}\le C_i \|\tilde f_i\|_{L^{p(\cdot)}(B_{r_i}(x_i))}\le C_i \| f\|_{L^{p(\cdot)}(\Omega)},\\
&\|G_i\|_{L^{\tilde q_i(\cdot)}(B_{\frac{r_i}{2}}(x_i))}\le C_i \left(\|\tilde f_i\|_{L^{p(\cdot)}(B_{r_i}(x_i))}+\|\tilde g_i\|_{L^{q(\cdot)}(B_{r_i}(x_i))}\right)\le C_i \|g\|_{L^{q(\cdot)}(\Omega)}.
\end{align*}

In particular, there exists a constant $C=C(\Omega,n,p,q)>0$ such that for all $i\in\{1,\dots,N\}$ we can find $S_i\in \R^{n\times n}_{\skw}$, $F_i\in L^{p(\cdot)}(B_{\frac{r_i}{2}}(x_i)\cap\Omega;\R^{n\times n})$, $G_i\in L^{q(\cdot)}(B_{\frac{r_i}{2}}(x_i)\cap\Omega;\R^{n\times n})$ satisfying
\begin{align}\label{onceagain}
\nabla u-S_i=F_i+G_i\quad\text{a.e.\ in $B_{\frac{r_i}{2}}(x_i)\cap\Omega$},
\end{align}
and
\begin{align}\label{onceagain2}
\|F_i\|_{L^{p(\cdot)}(B_{\frac{r_i}{2}}(x_i)\cap\Omega)}\le C \|f\|_{L^{p(\cdot)}(\Omega)}, \quad \|G_i\|_{L^{q(\cdot)}(B_{\frac{r_i}{2}}(x_i)\cap\Omega)}\le C \|g\|_{L^{q(\cdot)}(\Omega)}.
\end{align}
Let us set
\begin{align*}
&0<\alpha:=\min\Big\{|B_{\frac{r_i}{2}}(x_i)\cap B_{\frac{r_j}{2}}(x_j)\cap\Omega|\,:\,i,j\in\{1,\dots,N\},B_{\frac{r_i}{2}}(x_i)\cap B_{\frac{r_j}{2}}(x_j)\cap\Omega\neq\emptyset\Big\}.
\end{align*}
Notice that if $i,j\in{1,\dots,N}$ are such that $B_{r_i}(x_i)\cap B_{r_j}(x_j)\cap\Omega\neq \emptyset$, then~\eqref{onceagain}--\eqref{onceagain2} yield 
\begin{align*}
\alpha|S_i-S_j|&\le \|S_i-S_j\|_{L^1(B_{\frac{r_i}{2}}(x_i)\cap B_{\frac{r_j}{2}}(x_j)\cap\Omega)}\\
&\le \|\nabla u-S_i\|_{L^1(B_{\frac{r_i}{2}}(x_i)\cap\Omega)}+\|\nabla u-S_j\|_{L^1(B_{\frac{r_j}{2}}(x_j)\cap\Omega)}\\
&\le C\|g\|_{L^{q(\cdot)}(\Omega)}
\end{align*}
for a constant $C=C(\Omega,n,p,q)>0$. Let us define $S:=S_1$. Since $\Omega$ is connected, from the previous estimate we deduce that for all $i\in\{1,\dots,N\}$
\begin{align}\label{onceagain3}
|S-S_i|\le C\|g\|_{L^{q(\cdot)}(\Omega)}
\end{align}
for a constant $C=C(\Omega,n,p,q)>0$.
Let us define 
\begin{align*}
E_1:=B_{\frac{r_1}{2}}(x_1)\cap\Omega\quad\text{and}\quad E_i:=\left(B_{\frac{r_i}{2}}(x_i)\setminus\bigcup_{j=1}^{i-1}E_j\right)\cap\Omega\quad\text{for $i\in\{2,\dots,N$\}}.
\end{align*}
We define
$$F:=\sum_{i=1}^N\chi_{E_i}F_i\quad\text{and}\quad G:=\sum_{i=1}^N\chi_{E_i}(G_i+S_i-S).$$
Then, by~\eqref{onceagain}--\eqref{onceagain3} we conclude 
$$\nabla u-S=F+G\quad\text{a.e.\ in $\Omega$},$$
and 
\begin{align*}
&\|F\|_{L^{p(\cdot)}(\Omega)}\le \sum_{i=1}^N\|F_i\|_{L^{p(\cdot)}(E_i)}\le C\|f\|_{L^{p(\cdot)}(\Omega)},\\
&\|G\|_{L^{q(\cdot)}(\Omega)}\le \sum_{i=1}^N\|G_i\|_{L^{q(\cdot)}(E_i)}+\sum_{i=1}^N\|S_i-S\|_{L^{q(\cdot)}(E_i)}\le C\|g\|_{L^{q(\cdot)}(\Omega)},
\end{align*}
for a constant $C=C(\Omega,n,p,q)>0$. This shows the statement in~\eqref{eq:Spqkorn} and~\eqref{eq:Spq_est} hold. 

It remains to check that $S$ can be taken as $(\langle\nabla u\rangle_\Omega)_{\skw}$. Notice that by Proposition~\ref{prop:spq} and~\eqref{fromnowon}
\begin{align*}
|\Omega||S-(\langle\nabla u\rangle_\Omega)_{\skw}|\le \int_\Omega |\nabla u(x)-S|\,\de x&\le \|F\|_{L^1(\Omega)}+\|G\|_{L^1(\Omega)}\\
&\le C\|f\|_{L^{p(\cdot)}(\Omega)}+C\|g\|_{L^{q(\cdot)}(\Omega)}\le 2C\|g\|_{L^{q(\cdot)}(\Omega)}. 
\end{align*}
Hence,~\eqref{eq:Spqkorn}--\eqref{eq:Spq_est} hold with $S$ and $G$ replaced by $(\langle\nabla u\rangle_\Omega)_{\skw}$ and $G+S-(\langle\nabla u\rangle_\Omega)_{\skw}$, respectively.
\end{proof}

We can now prove the following rigidity result, which is the generalization of~\cite[Theorem~1.1]{CDM} to the case of variable exponents $p,q\in\mathcal P^{\log}_b(\Omega)$. In this case, we have to impose the restriction that $q$ is a fixed multiple of $p$, which however is enough in view of the application we present, that is the strong convergence of Theorem~\ref{thm:strconv}.

\begin{theorem}[Geometric rigidity for mixed growth and variable exponents]\label{thm:pxqxrig}
Let $\Omega\subset\R^n$ be a bounded domain with Lipschitz boundary. Let $p\in\mathcal P^{\log}_b(\Omega)$ be such that $p^->1$. Let $\mu\in[1,\infty)$ and define $q:=\mu p\in\mathcal P^{\log}_b(\Omega)$. Assume that $u\in W^{1,1}(\Omega;\R^n)$, $f\in L^{p(\cdot)}(\Omega)$, and $g\in L^{q(\cdot)}(\Omega)$ satisfy
\begin{equation}\label{eq:d-fg}
d(\nabla u(x),SO(n)) \le f(x)+g(x)\quad\text{for a.e.\ in $x\in\Omega$}. 
\end{equation}
There exist a constant $C=C(\Omega,n,p,\mu)>0$, a rotation $R\in SO(n)$, and two functions $F\in L^{p(\cdot)}(\Omega;\R^{n\times n})$ and $G\in L^{q(\cdot)}(\Omega;\R^{n\times n})$ satisfying
\begin{equation}\label{eq:Rpq-rig}
\nabla v(x)-R=F(x)+G(x)\quad\text{for a.e.\ $x\in \Omega$}, 
\end{equation}
and 
\begin{equation}\label{eq:Rpq_est}
\|F\|_{L^{p(\cdot)}(\Omega)}\le C\|f\|_{L^{p(\cdot)}(\Omega)}\quad\text{and}\quad\|G\|_{L^{q(\cdot)}(\Omega)}\le C\|g\|_{L^{q(\cdot)}(\Omega)}.
\end{equation}
 In particular, we can take $R\in SO(n)$ such that
\begin{equation}\label{eq:R-mean2}
|R-\langle\nabla u\rangle_\Omega|=d(\langle\nabla u\rangle_\Omega,SO(n)).
\end{equation}
\end{theorem}

As done in Proposition~\ref{prop:Q0pxrig}, we first prove the result for functions $u\in W^{1,1}(\Omega;\R^n)$ with uniformly bounded gradient. To this end, we follow the lines of~\cite[Lemma~3.1]{CDM}.

\begin{proposition}\label{prop:Mpxqxrig}
Let $\Omega\subset\R^n$ be a bounded domain with Lipschitz boundary. Let $p\in\mathcal P^{\log}_b(\Omega)$ be such that $p^->1$. Let $\mu\in[1,\infty)$ and define $q:=\mu p\in\mathcal P^{\log}_b(\Omega)$. Let $M>0$ and assume that $v\in W^{1,\infty} (\Omega;\R^n)$, $f\in L^{p(\cdot)}(\Omega)$, and $g\in L^{q(\cdot)}(\Omega)$ satisfy $\|\nabla v\|_{L^\infty(\Omega)}\le M$ and
$$d(\nabla v(x),SO(n))\le f(x)+g(x)\quad\text{for a.e.\ in $x\in\Omega$}.$$
There exist a constant $C=C(M,\Omega,n,p,\mu)>0$, a rotation $R\in SO(n)$, and two functions $F\in L^{p(\cdot)}(\Omega;\R^{n\times n})$ and $G\in L^{q(\cdot)}(\Omega;\R^{n\times n})$ satisfying
\begin{equation*}
\nabla v(x)-R=F(x)+G(x)\quad\text{for a.e.\ $x\in \Omega$}, 
\end{equation*}
and 
$$\|F\|_{L^{p(\cdot)}(\Omega)}\le C\|f\|_{L^{p(\cdot)}(\Omega)}\quad\text{and}\quad\|G\|_{L^{q(\cdot)}(\Omega)}\le C\|g\|_{L^{q(\cdot)}(\Omega)}.$$
\end{proposition}

\begin{proof}
First of all, without loss of generality we may assume that 
\begin{equation}\label{eq:fgM}
0\le f\le M+\sqrt{n}\quad\text{and}\quad 0\le g\le M+\sqrt{n}\quad\text{a.e.\ in $\Omega$}. 
\end{equation}
 Otherwise, we replace $f$ and $g$ by 
$$0\le f':=\min\{|f|,M+\sqrt{n}\}\le |f|\quad\text{and}\quad 0\le g':=\min\{|g|,M+\sqrt{n}\}\le |g|,$$
and we observe that 
$$d(\nabla v,SO(n))\le \min\{f+g,M+\sqrt{n}\}\le \min\{|f|+|g|,M+\sqrt{n}\}\le f'+g'\quad\text{a.e.\ in $\Omega$}.$$

Clearly, the geometric rigidity result is true for $\mu=1$ by Theorem~\ref{thm:pxrig}. If $\mu\in(1,\infty)$, there exists a unique $k\in\N$ such that $\mu\in(2^{k-1},2^k]$ and we prove the geometric rigidity by induction on $k\in\N$.

{\bf Step 1.} Let $k=1$ and let us prove the result for every $p\in\mathcal P^{\log}_b(\Omega)$ and $\mu\in(1,2]$. By Theorem~\ref{thm:pxrig} applied to $q\in\mathcal P^{\log}_b(\Omega)$ and using~\eqref{eq:fgM} there exists a constant $C=C(\Omega,n,p,\mu)>0$ and $O\in SO(n)$ such that by~\eqref{eq:fgM} 
\begin{equation}\label{eq:qxrigest}
\begin{aligned}
\|\nabla v-O\|_{L^{q(\cdot)}(\Omega)}&\le C\|d(\nabla v,SO(n))\|_{L^{q(\cdot)}(\Omega)}\le C\|f\|_{L^{q(\cdot)}(\Omega)}+C\|g\|_{L^{q(\cdot)}(\Omega)}\\
&=C\||f|^\mu\|_{L^{p(\cdot)}(\Omega)}^{\frac{1}{\mu}}+C\|g\|_{L^{q(\cdot)}(\Omega)}\le C(M+\sqrt{n})^{1-\frac{1}{\mu}}\|f\|_{L^{p(\cdot)}(\Omega)}^{\frac{1}{\mu}}+C\|g\|_{L^{q(\cdot)}(\Omega)}.
\end{aligned}
\end{equation}
Now, if it holds
$$\|f\|_{L^{p(\cdot)}(\Omega)}^{\frac{1}{\mu}}\le \|g\|_{L^{q(\cdot)}(\Omega)},$$
then the assertion follows by taking $F=0$ and $G=\nabla v- O $. If instead we have
\begin{equation}\label{eq:fgpq}
\|f\|_{L^{p(\cdot)}(\Omega)}^{\frac{1}{\mu}}>\|g\|_{L^{q(\cdot)}(\Omega)}, 
\end{equation}
then we consider the function $w(x):=O^Tv(x)$ for all $x\in\Omega$. By Taylor's expansion we have 
\begin{equation*}
d(A,SO(n))=|A_{\sym}-I|+O(|A-I|^2)\quad\text{for $|A-I|\to 0$},
\end{equation*}
 and thus we deduce 
\begin{equation}\label{eq:taylor}
\left|d(A,SO(n))-|A_{\sym}-I|\right|\le C|A-I|^2\quad\text{for all $A\in \R^{n\times n}$}
\end{equation}
for a constant $C=C(n)>0$. Hence there exists a constant $C=C(n)>0$ such that for a.e.\ $x\in\Omega$
\begin{equation}\label{eq:taylor2}
|ew(x)-I|\le d(\nabla w(x),SO(n))+C|\nabla w(x)-I|^2=d(\nabla v(x),SO(n))+C|\nabla v(x)-O|^2
\end{equation}
and, in view of~\eqref{eq:qxrigest} and~\eqref{eq:fgpq}, we have
\begin{equation}\label{eq:nablaO}
\begin{aligned}
\||\nabla v-O|^2\|_{L^{p(\cdot)}(\Omega)}&\le (M+\sqrt{n})^{2-\mu}\|\nabla v-O\|_{L^{q(\cdot)}(\Omega)}^\mu\\
&\le 2^{\mu-1}C^\mu(M+\sqrt{n})\|f\|_{L^{p(\cdot)}(\Omega)}+2^{\mu-1}C^\mu(M+\sqrt{n})^{2-\mu}\|g\|^\mu_{L^{q(\cdot)}(\Omega)}\\
&\le (2^{\mu-1}C^\mu(M+\sqrt{n})+2^{\mu-1}C^\mu(M+\sqrt{n})^{2-\mu})\|f\|_{L^{p(\cdot)}(\Omega)}.
\end{aligned}
\end{equation}
Let us define 
$$z(x):=w(x)-x,\quad \tilde f(x):=f(x)+C|\nabla v(x)-O|^2,\quad\tilde g(x):=g(x)\quad\text{for a.e.\ $x\in\Omega$}.$$
Then, by~\eqref{eq:taylor2} we get 
\begin{align}\label{ez}
|ez(x)|\le \tilde f(x)+\tilde g(x)\quad\text{for a.e.\ $x\in\Omega$}.
\end{align}
In particular, considering 
$$\tilde f':=\frac{\tilde f}{\tilde f+\tilde g}\chi_{\{\tilde f+\tilde g\ne 0\}}ez\quad\text{and}\quad \tilde g':=\frac{\tilde g}{\tilde f+\tilde g}\chi_{\{\tilde f+\tilde g\ne 0\}}ez$$
we have $|\tilde f'|\le \tilde f$ and $|\tilde g'|\le \tilde g$ a.e.\ in $\Omega$, and
$$ez(x)= \tilde f'(x)+\tilde g'(x)\quad\text{for a.e.\ $x\in\Omega$}.$$
By Theorem~\ref{thm:pxqxkorn} we get the existence of a constant $C=C(\Omega,n,p,\mu)>0$, a matrix $S\in \R^{n\times n}_{\skw}$, $\tilde F\in L^{p(\cdot)}(\Omega;\R^{n\times n})$, and $\tilde G\in L^{q(\cdot)}(\Omega;\R^{n\times n})$ such that 
$$\nabla z(x)-S=\nabla w(x)-I-S=\tilde F(x)+\tilde G(x)\quad\text{for a.e.\ $x\in\Omega$},$$
and, using also~\eqref{eq:nablaO} and~\eqref{ez}, we get 
$$\|\tilde F\|_{L^{p(\cdot)}(\Omega)}\le C\|\tilde f'\|_{L^{p(\cdot)}(\Omega)}\le C\|f\|_{L^{p(\cdot)}(\Omega)}\quad\text{and}\quad \|\tilde G\|_{L^{q(\cdot)}(\Omega)}\le C\|\tilde g'\|_{L^{q(\cdot)}(\Omega)}\le C\|g\|_{L^{q(\cdot)}(\Omega)}.$$
Let now $P\in SO(n)$ be such that $|I+S-P|=d(I+S,SO(n))$. Then, we have
$$|I+S-P|\le |I+S-\nabla w(x)|+d(\nabla w(x),SO(n))\le |\tilde F|+|\tilde G|+f+g.$$
By taking the norm in $L^{p(\cdot)}(\Omega)$ and using that $p(x)\le q(x)$ for a.e.\ $x\in\Omega$, we get
$$|I+S-P|\le C \|f\|_{L^{p(\cdot)}(\Omega)}+ C\|g\|_{L^{q(\cdot)}(\Omega)},$$
thanks to the previous estimates and H\"older's inequality in Proposition~\ref{prop:spq}. Therefore, setting $R:=OP\in SO(n)$, we can write for a.e.\ $x\in\Omega$
$$\nabla v(x)-R=O(\nabla w(x)-I-S)+O(I+S-P)=O\tilde F(x)+O\tilde G(x)+O(I+S-P).$$
 We now distinguish two cases. If we have
$$\|f\|_{L^{p(\cdot)}(\Omega)}\le \|g\|_{L^{q(\cdot)}(\Omega)},$$
we set $F:=O\tilde F$ and $G:=O(\tilde G+I+S-P)$, and we have
\begin{align*}
&\|F\|_{L^{p(\cdot)}(\Omega)}=\|\tilde F\|_{L^{p(\cdot)}(\Omega)}\le C\|f\|_{L^{p(\cdot)}(\Omega)},\\
&\| G \|_{L^{q(\cdot)}(\Omega)}\le \| \tilde{G} \|_{L^{q(\cdot)}(\Omega)}+\|I+S-P\|_{L^{q(\cdot)}(\Omega)}\le C\|g\|_{L^{q(\cdot)}(\Omega)}.
\end{align*}
Otherwise, for 
$$\|g\|_{L^{q(\cdot)}(\Omega)}<\|f\|_{L^{p(\cdot)}(\Omega)}$$
we set $F:=O(\tilde F+I+S-P)$ and $G:=O\tilde G$, and we have
\begin{align*}
&\|\tilde F\|_{L^{p(\cdot)}(\Omega)}\le \|F\|_{L^{p(\cdot)}(\Omega)}+\|I+S-P\|_{L^{p(\cdot)}(\Omega)}\le C\|f\|_{L^{p(\cdot)}(\Omega)},\\
&\|G\|_{L^{q(\cdot)}(\Omega)}=\|\tilde G\|_{L^{q(\cdot)}(\Omega)}\le C\|g\|_{L^{q(\cdot)}(\Omega)}.
\end{align*}
This proves the assertion for every $p\in\mathcal P^{\log}_b(\Omega)$ and $\mu\in(1,2]$.

{\bf Step 2.} Let $k\ge 2$ be fixed. Assume that the assertion is true for every $p\in\mathcal P^{\log}_b(\Omega)$ and $\mu\in (2^{k-2},2^{k-1}]$ and let us show it for every $p\in\mathcal P^{\log}_b(\Omega)$ and $\mu\in (2^{k-1},2^k]$.

We consider $2p\in\mathcal P^{\log}_b(\Omega)$ and $q=\mu p=\frac{\mu}{2}2p$. Since $\frac{\mu}{2}\in (2^{k-2},2^{k-1}]$ there exists a constant $C=C(M,\Omega,n,p,\mu)>0$, a rotation $O\in SO(n)$, and two functions $\hat F\in L^{2p(\cdot)}(\Omega;\R^{n\times n})$ and $\hat G\in L^{q(\cdot)}(\Omega;\R^{n\times n})$ satisfying 
$$\nabla v(x)-O=\hat F(x)+\hat G(x)\quad\text{for a.e.\ $x\in \Omega$},$$
and 
$$\|\hat F\|_{L^{2p(\cdot)}(\Omega)}\le C\|f\|_{L^{2p(\cdot)}(\Omega)}\quad\text{and}\quad\|\hat G\|_{L^{q(\cdot)}(\Omega)}\le C\|g\|_{L^{q(\cdot)}(\Omega)}.$$
Moreover, we may assume 
$$|\hat F|\le M+\sqrt{n}\quad\text{and}\quad |\hat G|\le M+\sqrt{n}\quad\text{a.e.\ in $\Omega$}.$$
Otherwise we replace $F$ and $G$ by 
$$\hat F':=\chi_{\Omega\setminus(A\cup B)}\hat F+\chi_A(\nabla v-O)\quad\text{and}\quad \hat G':=\chi_{\Omega\setminus(A\cup B)}\hat G+\chi_{B\setminus A}(\nabla v-O),$$
where
$$A:=\{x\in\Omega\,:\,|\hat F(x)|>M+\sqrt{n}\}\quad\text{and}\quad B:=\{x\in\Omega\,:\,|\hat G(x)|>M+\sqrt{n}\}.$$
We observe that $|\hat F'|\le \min\{|\hat F|,M+\sqrt{n}\}$ and $|\hat G'|\le \min\{|\hat G |,M+\sqrt{n}\}$ in a.e.\ in $\Omega$, and
$$\nabla v(x)-O=\hat F'(x)+\hat G'(x)\quad\text{for a.e.\ $x\in\Omega$}.$$
 Setting $w(x)=O^Tv(x)$ for a.e.\ $x\in\Omega$, by using Taylor's expansions~\eqref{eq:taylor} and arguing as in~\eqref{eq:taylor2} we get
$$|ew(x)-I|\le f(x)+g(x)+2C|\hat F(x)|^2+2C|\hat G(x)|^2\quad\text{for a.e.\ $x\in\Omega$}.$$
Let us define 
$$z(x):=w(x)-x,\quad \tilde f(x):=f(x)+2C|\hat F(x)|^2,\quad\tilde g(x):=g(x)+2C|\hat G(x)|^2\quad\text{for a.e.\ $x\in\Omega$}.$$
Then, we have
$$|ez(x)|\le \tilde f(x)+\tilde g(x)\quad\text{for a.e.\ $x\in\Omega$}$$
and
\begin{align*}
\|\tilde f\|_{L^{p(\cdot)}(\Omega)}&\le \|f\|_{L^{p(\cdot)}(\Omega)}+2C\|\hat F\|^2_{L^{2p(\cdot)}(\Omega)}\\
&\le\|f\|_{L^{p(\cdot)}(\Omega)}+2C\|f\|^2_{L^{2p(\cdot)}(\Omega)}\le C(M+\sqrt{n})\|f\|_{L^{p(\cdot)}(\Omega)}\\
\|\tilde g\|_{L^{q(\cdot)}(\Omega)}&\le \|g\|_{L^{q(\cdot)}(\Omega)}+2C\||\hat G|^2\|_{L^{q(\cdot)}(\Omega)},\\
&\le\|g\|_{L^{q(\cdot)}(\Omega)}+2C(M+\sqrt{n})\|\hat G\|_{L^{q(\cdot)}(\Omega)}\le C(M+\sqrt{n})\|g\|_{L^{q(\cdot)}(\Omega)}.
\end{align*}
By using the Korn inequality of Theorem~\ref{thm:pxqxkorn} and repeating the same argument used before for $\mu\in(1,2]$, we get the result for every $p\in\mathcal P^{\log}_b(\Omega)$ and $\mu\in(2^{k-1},2^k]$. By induction on $k\in\N$ we conclude.
\end{proof}

We can now prove Theorem~\ref{thm:pxqxrig} by using conditions (i) and (ii) of the Lusin approximation of Lemma~\ref{lem:lusin}, as done for Proposition~\ref{prop:Q0pxrig}. 

\begin{proof}[Proof of Theorem~\ref{thm:pxqxrig}]
Let $v\colon \Omega\to\R^n$ be the Lipschitz function given by the Lusin approximation of Lemma~\ref{lem:lusin} with $\lambda=2\sqrt{n}$ associated to the function $u$. Hence, there exists a constant $C=C(\Omega,n)>0$ such that 
\begin{equation}\label{eq:Cn}
\|\nabla v\|_{L^\infty(\Omega)}\le 2C\sqrt{n}. 
\end{equation}
Let us define
$$E:=\{x\in\Omega\,:\,u(x)\neq v(x)\}.$$ 
By arguing as in the proof of Proposition~\ref{prop:Q0pxrig} (see~\eqref{eq:set_est}), using property (ii) in Lemma~\ref{lem:lusin} and~\eqref{eq:d-fg}, we have
\begin{align}\label{eq:E-fg}
E\subseteq\{x\in\Omega\,:\,M_\Omega(d(\nabla u,SO(n)))(x)>\sqrt{n}\}\subseteq\{x\in\Omega\,:\,M_\Omega(f)(x)+M_\Omega(g)(x)>\sqrt{n}\}.
\end{align}
 Setting 
\begin{align*}
\tilde f:=f+(2C+1)M_\Omega (f)\quad\text{and}\quad \tilde g:=g+(2C+1)M_\Omega (g)\quad\text{a.e.\ in $\Omega$},
\end{align*}
where $C$ is the constant appearing in~\eqref{eq:Cn},
then by Proposition~\ref{prop:maxfun2} we deduce that $\tilde f\in L^{p(\cdot)}(\Omega)$, $\tilde g\in L^{q(\cdot)}(\Omega)$, and there exists a constant $C=C(\Omega,n,p,\mu)>0$ such that
$$\|\tilde f\|_{L^{p(\cdot)}(\Omega)}\le C\|f\|_{L^{p(\cdot)}(\Omega)}\quad\text{and}\quad\|\tilde g\|_{L^{q(\cdot)}(\Omega)}\le \|g\|_{L^{q(\cdot)}(\Omega)}.$$
 We claim that
\begin{equation}\label{eq:d-fg2}
d(\nabla v(x),SO(n))\le \tilde f(x)+\tilde g(x)\quad\text{for a.e.\ $x\in\Omega$}.
\end{equation}
Indeed, since $\nabla u=\nabla v$ a.e.\ on $\Omega\setminus E$, for a.e.\ $x\in\Omega\setminus E$ we have
$$d(\nabla v(x),SO(n))=d(\nabla u(x),SO(n))\le f(x)+g(x)\le \tilde f(x)+\tilde g(x).$$
Otherwise, for a.e.\ $x\in E$, by~\eqref{eq:Cn} and~\eqref{eq:E-fg} 
$$d(\nabla v(x),SO(n))\le (2C+1)\sqrt{n}\le (2C+1)M_\Omega (f)(x)+(2C+1)M_\Omega (g)(x)\le \tilde f(x)+\tilde g(x),$$
which proves~\eqref{eq:d-fg2}. By Proposition~\ref{prop:Mpxqxrig} we obtain a constant $C=C(\Omega,n,p,\mu)>0$, a rotation $R\in SO(n)$, and two functions $\tilde F\in L^{p(\cdot)}(\Omega;\R^{n\times n})$ and $\tilde G\in L^{q(\cdot)}(\Omega;\R^{n\times n})$ such that 
$$\nabla v(x)-R=\tilde F(x)+\tilde G(x)\quad\text{for a.e.\ $x\in\Omega$},$$
and
$$\|\tilde F\|_{L^{p(\cdot)}(\Omega)}\le C\|\tilde f\|_{L^{p(\cdot)}(\Omega)}\le C\|f\|_{L^{p(\cdot)}(\Omega)}\quad\text{and}\quad\|\tilde G\|_{L^{q(\cdot)}(\Omega)}\le C\|\tilde g\|_{L^{q(\cdot)}(\Omega)}\le C\|g\|_{L^{q(\cdot)}(\Omega)}.$$
Finally, thanks to~\eqref{eq:Cn} and~\eqref{eq:E-fg}, for a.e.\ $x\in\Omega$ we have
\begin{align*}
|\nabla u(x)-\nabla v(x)|&\le d(\nabla u(x),SO(n))+(2C+1)\sqrt{n}\chi_E(x) \le \tilde f(x)+\tilde g(x).
\end{align*}
Hence, by considering 
$$F:=\frac{\tilde f}{\tilde f+\tilde g}\chi_{\{\tilde f+\tilde g\ne 0\}}(\nabla u-\nabla v)+\tilde F\quad\text{and}\quad G:=\frac{\tilde g}{\tilde f+\tilde g}\chi_{\{\tilde f+\tilde g\ne 0\}}(\nabla u-\nabla v)+\tilde G,$$
 we get 
$$|F(x)|\le |\tilde f(x)|+|\tilde F(x)|\quad\text{and}\quad |G(x)|\le |\tilde g(x)|+|\tilde G(x)|\quad\text{for a.e.\ $x\in\Omega$}.$$
Therefore, $F\in L^{p(\cdot)}(\Omega;\R^{n\times n})$, $G\in L^{q(\cdot)}(\Omega;\R^{n\times n})$, and
$$\nabla u(x)-R=\nabla u(x)-\nabla v(x)+\nabla v(x)-R=F(x)+G(x)\quad\text{a.e.\ $x\in\Omega$}.$$
 This implies the rigidity result~\eqref{eq:Rpq-rig}--\eqref{eq:Rpq_est}.

Finally, if $R$, $F$, $G$ satisfies~\eqref{eq:Rpq-rig}--\eqref{eq:Rpq_est} and $\hat R\in SO(n)$ satisfies~\eqref{eq:R-mean2}, then by Proposition~\ref{prop:spq}
\begin{align*}
|\Omega||R-\hat R|\le \int_\Omega |\nabla u(x)-R|\,\de x&\le \|F\|_{L^1(\Omega)}+\|G\|_{L^1(\Omega)}\\
&\le C\|f\|_{L^{p(\cdot)}(\Omega)}+C\|g\|_{L^{q(\cdot)}(\Omega)}. 
\end{align*}
Hence, if $\|f\|_{L^{p(\cdot)}(\Omega)}\le \|g\|_{L^{q(\cdot)}(\Omega)}$, then~\eqref{eq:Rpq-rig}--\eqref{eq:Rpq_est} hold with $R$, $F$, and $G$ replaced by $\hat R$, $F$, and $G+R-\hat R$, respectively. On the contrary, if $\|f\|_{L^{p(\cdot)}(\Omega)}>\|g\|_{L^{q(\cdot)}(\Omega)}$, then~\eqref{eq:Rpq-rig}--\eqref{eq:Rpq_est} hold with $R$, $F$, and $G$ replaced by $\hat R$, $F+R-\hat R$, and $G$, respectively. In particular, the constant rotation $R$ in~\eqref{eq:Rpq-rig}--\eqref{eq:Rpq_est} depends only on $u$ and $\Omega$.

\end{proof}

As an application of Theorem~\ref{thm:pxqxrig}, we discuss that the equi-integrability on the right-hand side of~\eqref{eq:pxrig} transfers to equi-integrability on the left-hand side of~\eqref{eq:pxrig}. This corresponds to the generalization of ~\cite[Corollary~4.2]{CDM} to the setting of variable exponents. We recall the definition of equi-integrability for Lebesgue space with variable exponent. Let $E\subset\R^n$ be a measurable set with $|E|<\infty$ and let $p\in\mathcal P_b(E)$. We say that a family $\mathcal G\subset L^{p(\cdot)}(E;\R^d)$ is equi-integrable if for all $\eta>0$ there exists $M_\eta>0$ such that
$$\int_{\{x\in E\,:\,|g(x)|>M_\eta\}}|g(x)|^{p(x)}\,\de x<\eta\quad\text{for all $g\in\mathcal G$}.$$
Equivalently, by Proposition~\ref{prop:modular} the family $\mathcal G\subset L^{p(\cdot)}(E;\R^d)$ is equi-integrable if for all $\eta>0$ there exists $M_\eta>0$ such that
$$\|g\|_{L^{p(\cdot)}(\{x\in E\,:\,|g(x)|>M_\eta\})}<\eta\quad\text{for all $g\in\mathcal G$}.$$

\begin{remark}\label{rem:equi}
Notice that a family $\mathcal G\subset L^{p(\cdot)}(E;\R^d)$ is equi-integrable if and only if the family
$$\mathcal F:=\{|g(\cdot)|^{p(\cdot)}\,:\,g\in\mathcal G\}$$
is equi-integrable in $L^1(E)$. If $\mathcal G\subset L^{p(\cdot)}(E;\R^d)$ is equi-integrable, then $\mathcal G$ is bounded in $L^{p(\cdot)}(\Omega;\R^d)$. Moreover, a bounded family $\mathcal G\subset L^{p(\cdot)}(E;\R^d)$ is equi-integrable if for all $\eta>0$ there exists $\delta_\eta>0$ such that for all measurable sets $F\subseteq E$ with $|F|<\delta_\eta$ we have
\begin{align}\label{equi}
\int_F|g(x)|^{p(x)}\,\de x<\eta\quad\text{for all $g\in\mathcal G$},
\end{align}
see for instance~\cite[Proposition 3.1]{HRS}. Finally, if $\{g_j\}_j\subset L^{p(\cdot)}(E;\R^d)$ and $g_j\to g$ in $L^{p(\cdot)}(E;\R^d)$ as $j\to\infty$, then the family $\{g_j\}_j\subset L^{p(\cdot)}(E;\R^d)$ is equi-integrable.
\end{remark}

\begin{corollary}[Equi-integrability]\label{coro:equi}
Let $\Omega\subset\R^n$ be a bounded domain with Lipschitz boundary. Let $p\in\mathcal P^{\log}_b(\Omega)$ with $p^->1$. Let $\{\mu_j\}_j\subset(0,\infty)$, and $\{u_j\}_j\subset W^{1,p(\cdot)}(\Omega;\R^n)$ be such that the sequence
$$h_j:=\mu_j d(\nabla u_j,SO(n))\quad\text{a.e.\ in $\Omega$ for all $j\in\N$}$$
is equi-integrable in $L^{p(\cdot)}(\Omega)$. Then there exists a sequence of constant rotations $\{R_j\}_j \subset SO(n)$ such that the sequence
$$z_j:=\mu_j(\nabla u_j-R_j)\quad\text{a.e.\ in $\Omega$ for all $j\in\N$}$$
is equi-integrable in $L^{p(\cdot)}(\Omega;\R^{n\times n})$.
\end{corollary}

We remark that in the proof below it is crucial to choose $R_j$ independent of the parameter $\eta$ introduced therein. This is actually possible, as one may expect, but was not discussed in detail in~\cite{CDM}. For the sake of completeness, we give some clarifications on the issue. 

\begin{proof}
 Let us fix $\eta\in(0,1)$. Since the sequence $\{h_j\}_j$ is equi-integrable in $L^{p(\cdot)}(\Omega)$, there exists $M_\eta\ge 1$ such that
\begin{equation}\label{eq:hjeta}
\int_{\{x\in\Omega\,:\,|h_j(x)|>M_\eta\}}|h_j(x)|^{p(x)}\,\de x<\eta. 
\end{equation}
We define
$$E_j^\eta:=\{x\in\Omega\,:\,|h_j(x)|>M_\eta\},$$
and we consider
$$d(\nabla u_j ,SO(n))=d(\nabla u_j ,SO(n))\chi_{E_j^\eta}+d(\nabla u_j ,SO(n))\chi_{(E_j^\eta)^c }=:f_j^\eta+g_j^\eta\quad\text{a.e.\ in $\Omega$}.$$
By~\eqref{eq:hjeta} we have
\begin{align*}
\int_{\Omega}|\mu_jf_j^\eta(x)|^{p(x)}\,\de x=\int_{E_j^\eta}|h_j(x)|^{p(x)}\,\de x<\eta,
\end{align*}
and by Proposition~\ref{prop:modular} we derive
$$\|\mu_j f_j\|_{L^{p(\cdot)}(\Omega)}<\eta^{\frac{1}{p^+}}.$$
Moreover
\begin{align*}
\int_{\Omega}|\mu_j g_j(x)|^{2p(x)}\,\de x=\int_{(E^\eta_j)^c}\left|h_j(x)\right|^{2p(x)}\,\de x\le |\Omega| M_\eta^{2p^+},
\end{align*}
so that by Proposition~\ref{prop:modular} and the fact that $t^\frac{1}{q}\le 1+t$ for every $q\in [1,\infty)$ and $t\in[0,\infty)$, we obtain 
$$\|\mu_j g_j\|_{L^{2p(\cdot)}(\Omega)}\le 1+|\Omega|M_\eta^{2p^+}.$$
By Theorem~\ref{thm:pxqxrig} there exist a constant $C=C(\Omega,n,p) \ge 1 $, constant rotations $R_j\in SO(n)$, and functions $F_j^\eta\in L^{p(\cdot)}(\Omega;\R^{n\times n})$, and $G_j^\eta\in L^{2p(\cdot)}(\Omega;\R^{n\times n})$ such that
\begin{equation*}
\nabla u_j -R_j=F_j^\eta+G_j^\eta\quad\text{a.e.\ in $\Omega$},
\end{equation*}
and 
\begin{equation}\label{eq:FjGj}
\|\mu_j F_j^\eta\|_{L^{p(\cdot)}(\Omega)}\le C\eta^{\frac{1}{p^+}}\quad\text{and}\quad\|\mu_j G_j^\eta\|_{L^{2p(\cdot)}(\Omega)}\le C(1+|\Omega|M_\eta^{2p^+}). 
\end{equation}
We point out that we can take $R_j$ independent on $\eta$ thanks to~\eqref{eq:R-mean2}.
Let us consider the sequence of functions $\{z_j\}_j\subset L^1(\Omega;\R^{n\times n})$ defined by
$$z_j:=\mu_j(\nabla u_j -R_j)=\mu_jF_j^\eta+\mu_jG_j^\eta\quad\text{a.e.\ in $\Omega$ for all $j\in\N$}.$$
By construction we have
\begin{equation}\label{eq:muj_est}
|z_j(x)|^{p(x)}\le 2^{p^+}\left(|\mu_jF_j^\eta(x)|^{p(x)}+|\mu_jG_j^\eta(x)|^{p(x)}\right). 
\end{equation}
For all measurable sets $E\subseteq\Omega$, by H\"older's inequality, Proposition~\ref{prop:modular}, and~\eqref{eq:FjGj}--\eqref{eq:muj_est} we obtain
\begin{equation*}
\begin{aligned}
\int_E |z_j(x)|^{p(x)}\,\de x&\le 2^{p^+}\int_{E}|\mu_jF_j^\eta(x)|^{p(x)}\,\de x+2^{p^+}\int_{E}|\mu_jG_j^\eta(x)|^{p(x)}\,\de x\\
&\le 2^{p^+}\int_{\Omega}|\mu_jF_j^\eta(x)|^{p(x)}\,\de x+2^{p^+}\left(\int_{\Omega}|\mu_jG_j^\eta(x)|^{2p(x)}\,\de x\right)^{\frac{1}{2}}|E|^{\frac{1}{2}}\\
&\le 2^{p^+}\max\{\|\mu_j F_j^\eta\|^{p^-}_{L^{p(\cdot)}(\Omega)},\|\mu_j F_j^\eta\|^{p^+}_{L^{p(\cdot)}(\Omega)}\}\\
&\quad+2^{p^+}\max\{\|\mu_j G_j^\eta\|^{p^-}_{L^{2p(\cdot)}(\Omega)},\|\mu_j G_j^\eta\|^{p^+}_{L^{2p(\cdot)}(\Omega)}\} |E|^{\frac{1}{2}}\\
&\le C\left(\eta^\frac{p^-}{p^+}+(1+M_\eta^{2(p^+)^2})|E|^\frac{1}{2}\right)
\end{aligned}
\end{equation*}
for a constant $C=C(\Omega,n,p) \ge 1 $. This implies the equi-integrability of the sequence $\{z_j\}_j$ in $L^{p(\cdot)}(\Omega;\R^{n\times n})$. In fact, we choose $\eta$ such that $C\eta^\frac{p^-}{p^+}< \frac{\theta}{2}$. Then, for all measurable sets $E\subseteq\Omega$ with $|E| <\delta_\theta$, where $C(1+M_\eta^{2(p^+)^2})\delta_\theta^\frac{1}{2}< \frac{\theta}{2}$, we get~\eqref{equi} with $\theta$ in place of $\eta$. 
\end{proof}

\section{$\Gamma$-convergence of finite elasticity energies with variable exponent growth}\label{sec:4}

In this section we extend the $\Gamma$-convergence result of~\cite{AgDMDS} to the case of a variable exponent $p$, see Theorem~\ref{thm:gamma_con} and Theorem~\ref{thm:strconv}. We start by introducing the setting of~\cite{AgDMDS} for variable exponents. 
Let $\Omega\subset\R^n$ be a bounded domain with Lipschitz boundary and let $p\in\mathcal P_b(\Omega)$ with 
$$1<p^-\le p^+\le 2.$$
We consider the function $g\colon [1,2]\times [0,\infty)\to [0,\infty)$ introduced in~\cite[Section 2]{AgDMDS}, which is defined as
\begin{equation}\label{eq:g}
g(q,t):=
\begin{cases}
\frac{t^2}{2}&\text{for $t\in[0,1]$ and $q\in[1,2]$},\\
\frac{t^{q}}{q}+\frac{1}{2}-\frac{1}{q}&\text{for $t\in(1,\infty)$ and $q\in[1,2]$}.
\end{cases}
\end{equation}
 We consider hyperelastic energies of the form
$$v \mapsto \int_\Omega W(x,\nabla v) \, \de x, $$
where $v\colon \Omega \to \R^n$ denotes the \emph{deformation} and $W\colon \Omega\times \R^{n\times n}\to [0,\infty]$ is a measurable \emph{stored energy density} satisfying the following properties for a.e.\ $x\in\Omega$:
\begin{itemize}
 \item[(i)] $W(x,\cdot)$ is frame indifferent, i.e., $W(x,RF) = W(x,F)$ for all $F \in \R^{n \times n}$ and $R \in SO(n)$;
 \item[(ii)] $W(x,\cdot)$ is of class $C^2$ in some neighbourhood of $SO(n)$, independent of $x$, where the second derivatives are bounded by a constant independent of $x$;
 \item[(iii)] $W(x,F)=0$ if $F\in SO(n)$;
 \item[(iv)] $W(x,F)\ge g(p(x),d(F,SO(n)))$.
\end{itemize}
We are interested in the regime of small deformations and prescribe boundary conditions for \emph{rescaled displacement fields} $u(x) = \frac{v(x)-x}{\varepsilon}$ for $x \in \Omega$, where the small parameter $\varepsilon>0$ represents the order of the strain. More precisely, we prescribe a Dirichlet condition $h\in W^{1,\infty}(\Omega;\R^n)$ on a part $\partial_D\Omega$ of $\partial\Omega$ with Lipschitz boundary in $\partial\Omega$, according to~\cite[Definition 2.1]{AgDMDS}, represented by the subset 
$$W^{1,p(\cdot)}_h(\Omega;\R^n):=\big\{u\in W^{1,p(\cdot)}(\Omega;\R^n)\,:\,u=h\text{ $\mathcal H^{n-1}$-a.e.\ on $\partial_D\Omega$}\big\}.$$
 Here, the equality refers to the traces of the functions on the boundary $\partial \Omega$. Accordingly, suitably rescaled energy functionals $\mathcal F_\varepsilon \colon W^{1,p(\cdot)}(\Omega;\R^n)\to [0,\infty]$ are defined by
$$\mathcal F_\varepsilon(u):=
\begin{cases}
\frac{1}{\varepsilon^2}\int_\Omega W(x,I+\varepsilon\nabla u(x))\,\de x&\text{if $u\in W^{1,p(\cdot)}_h(\Omega;\R^n)$},\\
\infty&\text{otherwise},
\end{cases}
$$
 for $\varepsilon >0$. Our goal is to identify the functional $\mathcal F \colon W^{1,p(\cdot)}(\Omega;\R^n)\to [0,\infty]$ given by 
$$\mathcal F(u):=
\begin{cases}
 \frac{1}{2} \int_\Omega D^2W(x,I)[eu(x)]^2\,\de x&\text{if $u\in W^{1,2}_h(\Omega;\R^n)$},\\
\infty&\text{otherwise},
\end{cases}
$$
 as the effective energy in the small-strain limit $\varepsilon \to 0$. Here $D^2W(x,I)[M]^2$ denotes the second derivative of $W(x,\cdot)$ applied to the pair $[M,M]$ for $M \in \R^{n \times n}$. By assumptions (i), (ii), and (iv) on $W$, $D^2W(x,I)[\cdot]^2$ is positive definite on symmetric matrices and vanishes on skew-symmetric matrices. 

We obtain the following compactness and $\Gamma$-convergence results.

\begin{proposition}[Compactness]\label{lem:equicoerv} 
Let $\Omega\subset\R^n$ be a bounded domain with Lipschitz boundary and let $p\in\mathcal P_b(\Omega)$ with 
$$1<p^-\le p^+\le 2.$$
Assume that $W$ satisfies \emph{(i)--(iv)}. There exists a constant $C=C(\Omega,n)>0$ such that for all $\{u_\varepsilon\}_{\varepsilon\in(0,1)}\subset W_h^{1,p(x)}(\Omega;\R^n)$ we have
\begin{equation}\label{eq:equicoer}
\int_\Omega|\nabla u_\varepsilon(x)|^{p(x)} \,\de x \le C\left[1+\mathcal F_\varepsilon(u_\varepsilon)+\left(\int_{\partial_D\Omega}|h|\,\de\mathcal H^{n-1}\right)^2\right]\quad\text{for all $\varepsilon\in(0,1)$}.
\end{equation}
\end{proposition}

\begin{theorem}[$\Gamma$-convergence]\label{thm:gamma_con}
Let $\Omega\subset\R^n$ be a bounded domain with Lipschitz boundary and let $p\in\mathcal P_b(\Omega)$ with 
$$1<p^-\le p^+\le 2.$$
Assume that $W$ satisfies \emph{(i)--(iv)}. For all sequence $\varepsilon_j\to 0$ as $j\to\infty$ we have
$$\mathcal F_{\varepsilon_j}\xrightarrow{\Gamma} \mathcal F\quad\text{as $j\to\infty$}$$
in the weak topology of $W^{1,p(\cdot)}_h(\Omega;\R^n)$. 
\end{theorem}

As usual in the theory of $\Gamma$-convergence, the compactness and $\Gamma$-convergence results imply convergence of minima and minimizers. The proofs of Proposition~\ref{lem:equicoerv} and Theorem~\ref{thm:gamma_con} are rather straightforward adaptations of corresponding results in~\cite{AgDMDS}, and are thus postponed to the appendix. In contrast to the previous sections, their arguments rely on a simpler rigidity estimate, see Theorem~\ref{thm:gpxrig}, and require only that $p$ is measurable.

As a refinement of the result, we use the rigidity estimate with mixed growth conditions of Theorem~\ref{thm:pxqxrig} to show the strong convergence in $W^{1,p(\cdot)}(\Omega;\R^n)$ of the minimizers of $\mathcal F_\varepsilon$ towards the minimizers of $\mathcal F$. For this part, we need to assume that $p$ is $\log$-Hölder continuous, more precisely that $p\in\mathcal P^{\log}_b(\Omega)$ with
$$1<p^-\le p^+\le 2.$$
We recall the definition of equi-integrability introduced in Section~\ref{sec:3}. Thanks to Remark~\ref{rem:equi} and Vitali's convergence theorem, we deduce the following convergence result. 

\begin{proposition}[Vitali's convergence theorem]
If $\{g_j\}_j\subset L^{p(\cdot)}(\Omega;\R^d)$ is a sequence of equi-integrable functions in $L^{p(\cdot)}(\Omega;\R^d)$ and $g_j\to g$ in measure on $\Omega$ as $j\to\infty$, then $g\in L^{p(\cdot)}(\Omega;\R^d)$ and $g_j\to g$ in $L^{p(\cdot)}(\Omega;\R^d)$.
\end{proposition}

Moreover, we will make use of the following result.

\begin{lemma}\label{lem:poincare2}
Let $\Omega\subset\R^n$ be a bounded domain with Lipschitz boundary. There exists a constant $C=C(\Omega,n)>0$ such that for every $\varepsilon>0$, $R\in SO(n)$, and $u\in W^{1,1}_h(\Omega;\R^n)$ we have
\begin{equation*}
|I-R|\le C\left(\int_\Omega|I+\varepsilon\nabla u(x)-R|\,\de x+\varepsilon\int_{\partial_D\Omega}h\,\de\mathcal H^{n-1}\right). 
\end{equation*}
\end{lemma}

\begin{proof}
The proof is a consequence of~\cite[Lemma 3.2]{AgDMDS}, Poincaré's inequality, and the continuity of the trace operator, see also~\cite[Equations (3.4) and (3.5)]{AgDMDS}.
\end{proof}

Finally, we also recall the following property that holds for recovery sequences of $\mathcal F_\varepsilon$, whose proof can be deduced from~\cite{AgDMDS}.

\begin{proposition}
Let $\Omega\subset\R^n$ be a bounded domain with Lipschitz boundary and let $p\in\mathcal P_b(\Omega)$ with 
$$1<p^-\le p^+\le 2.$$
Assume that $W$ satisfies \emph{(i)--(iv)}. Let $\varepsilon_j\to 0$ as $j\to\infty$ and let $\{u_j\}_j$ be a recovery sequence of $\mathcal F_{\varepsilon_j}$ and $\mathcal{F}$ for $u\in W^{1,2}_h(\Omega;\R^n)$, that is $u_j\rightharpoonup u$ weakly in $W^{1,p(\cdot)}(\Omega;\R^n)$ and $\mathcal F_{\varepsilon_j}(u_j)\to\mathcal F(u)$ as $j\to\infty$. Define
\begin{equation}\label{eq:Bj}
B_j:=\Big\{x\in\Omega\,:\, \varepsilon_j^\frac{1}{2}|\nabla u_j(x)|\le 1\Big\}. 
\end{equation}
Then
\begin{align}
&eu_j\chi_{B_j}\to eu\quad\text{in $L^2(\Omega;\R^{n\times n}_{\sym})$ as $j\to\infty$},\label{eq:euj}\\
&\frac{1}{\varepsilon_j^2}\int_{B_j^c}W(I+\varepsilon_j\nabla u_j(x))\,\de x\to 0\quad\text{as $j\to\infty$}.\label{eq:wuj}\\
& \text{$u_j\to u$ strongly in $W^{1,p^-}(\Omega;\R^n)$}\quad\text{as $j\to\infty$}.\label{eq:strong-wuj}
\end{align}

\end{proposition}

\begin{proof}
Let $\varepsilon_j\to 0$ as $j\to\infty$ and let $\{u_j\}_j$ be a recovery sequence of $\mathcal F_{\varepsilon_j}$ and $\mathcal{F}$ for $u\in W^{1,2}_h(\Omega;\R^n)$. 
Consider the extensions (with value $+\infty$) $\hat{\mathcal F}_\varepsilon,\hat{\mathcal F}\colon W^{1,p^-}(\Omega;\R^n)\to [0,\infty]$ of $\mathcal F_\varepsilon$ and $\mathcal F$, respectively. Then $\{u_j\}_j$ is a recovery sequence of $\hat{\mathcal{F}}_{\varepsilon_j}$ and $\hat{\mathcal{F}}$ for $u\in W^{1,2}_h(\Omega;\R^n)$. By the definition of $g$ it follows that $g(p(x),t)\ge g(p^-,t)$ for a.e.\ $x\in \Omega$ and every $t\in[0,\infty)$. In particular, since $1<p^-\le 2$, we can apply~\cite[Theorem~2.5]{AgDMDS} to deduce~\eqref{eq:strong-wuj}. Finally,~\eqref{eq:euj} and~\eqref{eq:wuj} can be deduced from the proof of~\cite[Theorem~2.5]{AgDMDS}, see in particular~\cite[Equations~(5.7) and~(5.13)]{AgDMDS}. 
\end{proof}

We can now prove the strong convergence in $W^{1,p(\cdot)}(\Omega;\R^n)$ of the minimizers of $\mathcal F_\varepsilon$ towards the minimizers of $\mathcal F$.

\begin{theorem}[Recovery sequences]\label{thm:strconv}
Let $\Omega\subset\R^n$ be a bounded domain with Lipschitz boundary and let $p\in\mathcal P^{\log}_b(\Omega)$ with 
$$1<p^-\le p^+\le 2.$$
Assume that $W$ satisfies \emph{(i)--(iv)}. Let $\varepsilon_j\to 0$ as $j\to\infty$ and let $\{u_j\}_j$ be a recovery sequence of $\mathcal F_{\varepsilon_j}$ and $\mathcal F$ for $u\in W^{1,2}_h(\Omega;\R^n)$. Then $u_j\to u$ strongly in $W^{1,p(\cdot)}(\Omega;\R^n)$ as $j\to\infty$.
\end{theorem}

\begin{proof}
Let $\varepsilon_j\to 0$ as $j\to\infty$ and let $\{u_j\}_j$ be a recovery sequence of $\mathcal F_{\varepsilon_j}$ and $\mathcal{F}$ for $u\in W^{1,2}_h(\Omega;\R^n)$. By~\eqref{eq:strong-wuj} we get $u_j\to u$ strongly in $W^{1,p^-}(\Omega;\R^n)$. Hence, to conclude it is enough to show that the sequence $\{\nabla u_j\}_j$ is equi-integrable in $L^{p(\cdot)}(\Omega;\R^{n\times n})$ and to apply Vitali's convergence theorem. 
%Indeed, up to a subsequence we know that
% $$\nabla u_j(x)\to \nabla u(x)\quad\text{for a.e.\ $x\in\Omega$ as $j\to\infty$}.$$
% Hence, if the sequence $\{|\nabla u_j(\cdot)|^{p(\cdot)}\}_j$ is equi-integrable, then the sequence $g_j(x):=|\nabla u_j(x)-\nabla u(x)|^{p(x)}$ for a.e.\ $x\in\Omega$ and $j\in\N$ would satisfy the assumptions of Vitali's convergence theorem with $g=0$. In particular, $g_j\to 0$ in $L^1(\Omega)$ as $j\to\infty$, which is equivalent to $\nabla u_j\to \nabla u$ strongly in $L^{p(\cdot)}(\Omega;\R^{n\times n})$ as $j\to\infty$.

{\bf Step 1.} We may assume that $\varepsilon_j\in(0,1)$ for all $j\in\N$ and we define $v_j(x):=x+\varepsilon_ju_j(x)$ for a.e.\ $x\in\Omega$. In order to prove that the sequence $\{\nabla u_j\}_j$ is equi-integrable in $L^{p(\cdot)}(\Omega;\R^{n\times n})$, we first show that the sequence
$$h_j(x):=\frac{d(\nabla v_j(x),SO(n))}{\varepsilon_j}\quad\text{for a.e.\ $x\in\Omega$ and $j\in\N$}$$
is equi-integrable in $L^{p(\cdot)}(\Omega)$. For all $j\in\N$ we consider the set $B_j$ defined in~\eqref{eq:Bj}, and we write
\begin{equation}\label{eq:fj}
h_j=h_j\chi_{B_j}+h_j\chi_{B_j^c}\quad\text{a.e.\ in $\Omega$}.
\end{equation}
Our goal is to show that both terms in~\eqref{eq:fj} are equi-integrable. By Taylor's expansion~\eqref{eq:taylor}, for a.e.\ $x\in\Omega$ we have 
\begin{align*}
d(\nabla v_j(x),SO(n))\le \varepsilon_j|eu_j(x)|+C\varepsilon_j^2|\nabla u_j(x)|^2
\end{align*}
for a constant $C=C(n) \ge 1 $. Therefore, as $p^+\le 2$, for a.e.\ $x\in\Omega$ we deduce
\begin{align*}
h_j(x)^{p(x)}\chi_{B_j}(x)\le 2|eu_j(x)|^{p(x)}\chi_{B_j}(x)+ 2C^2 \varepsilon_j^{p(x)}|\nabla u_j(x)|^{2p(x)}\chi_{B_j}(x).
\end{align*}
Then, the definition of $B_j$ implies that for a.e.\ $x\in\Omega$ it holds 
\begin{align*}
h_j(x)^{p(x)}\chi_{B_j}(x)\le 2|eu_j(x)|^{p(x)}\chi_{B_j}(x)+ 2 C^2 \varepsilon_j^{\frac{p(x)}{2}}|\nabla u_j(x)|^{p(x)}\chi_{B_j}(x).
\end{align*}
By Proposition~\ref{lem:equicoerv} the sequence $\{\nabla u_j\}_j$ is bounded in $L^{p(\cdot)}(\Omega;\R^{n\times n})$, which implies that the sequence $\{ \varepsilon_j^{\frac{1}{2}} \nabla u_j\chi_{B_j}\}_j$ converges to zero in $L^{p(\cdot)}(\Omega;\R^{n\times n})$ and is thus equi-integrable. Moreover, by~\eqref{eq:euj} we have that the sequence $\left\{eu_j\chi_{B_j}\right\}_j$ is equi-integrable in $L^2(\Omega;\R^{n\times n}_{\sym})$, that is for all $\eta>0$ there exists $M_\eta \ge 1 $ such that 
$${\int_{\{x\in B_j\,:\,|eu_j(x)|>M_\eta\}}|eu_j(x)|^2\,\de x<\eta \quad \text{for all $j \in \N$.}}$$
Thus, due to $M_\eta \ge 1$ and $p^+\le 2$, also the sequence $\{eu_j\chi_{B_j}\}_j$ is equi-integrable in $L^{p(\cdot)}(\Omega;\R^{n\times n}_{\sym})$, since
\begin{align*}
&\int_{\{x\in B_j\,:\,|eu_j(x)|> M_\eta \}}|eu_j(x)|^{p(x)}\,\de x\le \int_{\{x\in B_j\,:\,|eu_j(x)|>M_\eta\}}|eu_j(x)|^2\,\de x<\eta.
\end{align*}
This implies that the first term in~\eqref{eq:fj} is equi-integrable. 

For the second term, we observe that one can check $(t/\delta)^q\le 2g(q,t)/\delta^2+1$ for all $q\in[1,2]$, $t\in[0,\infty)$, and $\delta \in(0,1)$. Then, we use (iv) to obtain the estimate 
$$h_j(x)^{p(x)}\le \frac{2}{\varepsilon_j^2}g(p(x),d(\nabla v_j(x),SO(n)))+1\le \frac{2}{\varepsilon_j^2}W(\nabla v_j(x))+1.$$ 
Thus, by~\eqref{eq:wuj}, Chebyshev's inequality, and the fact that the sequence $\{\nabla u_j\}_j$ is bounded in $L^1(\Omega;\R^{n\times n})$, we conclude that
$$\int_{B_j^c} h_j(x)^{p(x)}\,\de x \le \frac{ 2 }{\varepsilon_j^2}\int_{B_j^c}W(\nabla v_j(x))\,\de x+|B_j^c|\to 0\quad\text{as $j\to\infty$}.$$
Hence $h_j\chi_{B_j^c}\to 0$ in $L^{p(\cdot)}(\Omega)$ as $j\to\infty$, which gives that $\{h_j\chi_{B_j^c}\}_j$ is equi-integrable in $L^{p(\cdot)}(\Omega)$.

{\bf Step 2.} By Corollary~\ref{coro:equi} for every $j\in\N$ there exists a constant rotation $R_j\in SO(n)$ such that the sequence
$$z_j(x):=\frac{\nabla v_j(x)-R_j}{\varepsilon_j}\quad\text{for a.e.\ $x\in\Omega$ and for all $j\in\N$}$$
is equi-integrable in $L^{p(\cdot)}(\Omega;\R^{n\times n})$. In particular, the sequence $\{z_j\}_j$ is bounded in $L^{p(\cdot)}(\Omega;\R^{n\times n})$. Moreover, by Lemma~\ref{lem:poincare2} there exists a constant $C=C(\Omega,n)>0$ such that
$$ |I-R_j|\le C\left(\int_\Omega|\nabla v_j(x)-R_j|\,\de x+\varepsilon_j\int_{\partial_D\Omega}|h|\,\de\mathcal H^{n-1}\right).$$
Hence, by using the inequality $t\le 1+t^q$ for every $q\in[1,2]$ and $t\in[0,\infty)$ and the boundedness of $\lbrace z_j \rbrace_j$ in $L^{p(\cdot)}(\Omega;\R^{n\times n})$ we deduce the following estimate
\begin{align*}
\frac{|I-R_j|}{\varepsilon_j}&\le C\left(\int_\Omega\left|\frac{\nabla v_j(x)-R_j}{\varepsilon_j}\right|\,\de x+\int_{\partial_D\Omega}|h|\,\de\mathcal H^{n-1}\right)\\
&=C\left(\int_\Omega|z_j(x)|\,\de x+\int_{\partial_D\Omega}|h|\,\de\mathcal H^{n-1}\right)\le C
\end{align*}
for a constant $C=C(\Omega,n,h )\ge 1$. Hence, recalling $v_j(x) = x + \varepsilon_j u_j(x)$, for all measurable sets $E\subseteq\Omega$ we get
\begin{align*}
\int_E |\nabla u_j(x)|^{p(x)}\,\de x&\le 2\int_E |z_j(x)|^{p(x)}\,\de x+2\int_E \left|\frac{I-R_j}{\varepsilon_j}\right|^{p(x)}\,\de x\le 2\int_E |z_j(x)|^{p(x)}\,\de x+2C^2|E|.
\end{align*}
From this we deduce the equi-integrability of the sequence $\{\nabla u_j\}_j$ in $L^{p(\cdot)}(\Omega;\R^{n\times n})$.
\end{proof}

\section{Appendix}

In this appendix, we give the proof of the auxiliary results used in Sections~\ref{sec:2}--\ref{sec:3} and of the $\Gamma$-convergence result of Section~\ref{sec:4}.

\subsection{Proofs of the auxiliary results}

 We start with the proof of the weighted Poincaré inequality in $W^{1,p(\cdot)}(\Omega;\R^d)$. 

\begin{proof}[Proof of Proposition~\ref{prop:poincare}]
It is enough to prove~\eqref{eq:poincare} in the case $d=1$. We follow the argument used in~\cite[Theorem~3.1]{DD}. We fix $f\in L^{p(\cdot)}(\Omega)$. Proceeding as in~\cite[Equation (3.2)]{DD} there exists $a\in\R$ and a constant $C=C(\Omega,n)>0$ such that 
$$\int_{\Omega}|f(y)-a||g(y)| \, \de y \le C\int_\Omega\left(\int_{|x-y|\le Cd(x,\partial\Omega)}\frac{|g(y)|}{|x-y|^{n-1}}\,\de y\right)|\nabla f(x)|\,\de x$$
for every function $g\in C_c^\infty(\Omega)$. By~\cite[Lemma~2.8.3]{Z} for every $x\in\R^n$ we have
$$\int_{|x-y|\le Cd(x,\partial\Omega)}\frac{|g(y)|}{|x-y|^{n-1}}\,\de y\le Cd(x,\partial\Omega)M(g)(x)$$
for a constant $C=C(\Omega,n)>0$, where $g$ is extended trivially outside of $\Omega$. Hence, thanks to Hölder's inequality in Proposition~\ref{prop:spq} and Proposition~\ref{prop:maxfun} 
\begin{align*}
\int_{\Omega}|f(y)-a||g(y)|\, \de y &\le C\int_\Omega M(g)(x)d(x,\partial\Omega)|\nabla f(x)|\,\de x\\
&\le C\|M(g)\|_{L^{p'(\cdot)}(\Omega)}\|d(\cdot,\partial\Omega)\nabla f\|_{L^{p(\cdot)}(\Omega)}\le C\|g\|_{L^{p'(\cdot)}(\Omega)}\|d(\cdot,\partial\Omega)\nabla f\|_{L^{p(\cdot)}(\Omega)}
\end{align*}
for a constant $C=C(\Omega,n,p)>0$. By~\cite[Corollary~3.4.13]{DHHR} we conclude that
\begin{align*}
\|f-a\|_{L^{p(\cdot)}(\Omega)}\le 2\sup_{g\in C_c^\infty(\Omega),\|g\|_{L^{p'(\cdot)}(\Omega)}\le 1}\int_{\Omega}|f(y)-a||g(y)|\, \de y \le C\|d(\cdot,\partial\Omega)\nabla f\|_{L^{p(\cdot)}(\Omega)}
\end{align*}
for a constant $C=C(\Omega,n,p)>0$. 
\end{proof}

We conclude this subsection with the proof of the extension result of Theorem~\ref{thm:extension}.

\begin{proof}[Proof of Theorem~\ref{thm:extension}]

In the proof, for convenience we replace the ball $B_R$ with the cube $Q_R:=(-R,R)^n$. It is clear that if the result holds for cubes, then it holds also for balls.

We introduce a regularization of the distance function $\delta\colon \R^n\setminus\overline\Omega\to [0,\infty)$, which satisfies
\begin{equation}\label{eq:delta}
0<2(x_n-\varphi(x'))\le \delta(x)\le c_1(x_n-\varphi(x'))\quad\text{for every $x\in \R^n\setminus\overline\Omega$}, 
\end{equation}
for a constant $c_1=c_1(L,n) \ge 2 $, see~\cite[Chapter 4.6.2]{stein}. The function $\delta$ lies in $C^2(\R^n\setminus\overline\Omega)$ and its derivatives satisfy
\begin{equation}\label{eq:der_delta}
|\nabla \delta(x)|\le c_2\quad\text{and}\quad |\nabla^2\delta (x)|\le \frac{c_2}{\delta (x)}\quad\text{for every $x\in\R^n\setminus\Omega$},
\end{equation}
where $c_2=c_2(L,n) \ge 1 $. In particular, since $\delta$ vanishes on $\partial\Omega$, by setting $\delta=0$ on $\overline\Omega$ we can extend $\delta$ to a Lipschitz function on $\R^n$ which still satisfies $|\nabla \delta(x)|\le c_2$ for a.e.\ $x\in\R^n$.

We fix $R>0$ and we choose $r=r(R,L,n)>0$ such that
\begin{equation}\label{eq:r}
r<\frac{R}{2c_1(1+L)}.
\end{equation}
By~\eqref{eq:delta} we deduce that
\begin{align}\label{lastpage}
(x',x_n-\lambda\delta(x))\in Q_R\cap\Omega\quad\text{for every $x=(x',x_n)\in Q_r\setminus\overline\Omega$ and for every $\lambda\in[1,2]$}.
\end{align}
We fix $p,q,u,f,g$ as in the assumptions of the theorem and we consider a function $\psi\in C^1(\R)$ which satisfies
$$\int_1^2 \psi (\lambda)\,\de\lambda=1\quad\text{and}\quad\int_1^2\lambda\psi(\lambda)\,\de\lambda=0.$$
We define $\tilde u\colon Q_r\to\R^n$ as 
\begin{equation*}
\tilde u(x):=\int_1^2\psi(\lambda)[u(x',x_n-\lambda\delta(x))-\lambda\nabla \delta(x)u_n(x',x_n-\lambda\delta(x))]\,\de\lambda\quad\text{for a.e.\ $x\in Q_r$}.
\end{equation*}
We have that $\tilde u=u$ a.e.\ on $Q_r\cap\Omega$, and in view of the computations done in~\cite{nitsche,CDM} we deduce that $\tilde u\in W^{1,1}(Q_r;\R^n)$ and
\begin{align*}
e\tilde u(x)=&\int_1^2\psi(\lambda)[eu(x',x_n-\lambda\delta(x))+\lambda^2(eu)_{nn}(x',x_n-\lambda\delta(x))\nabla \delta (x)\otimes \nabla \delta (x) )]\,\de\lambda\\
&-\int_1^2\lambda\psi(\lambda)[eu(x',x_n-\lambda\delta(x))e_n\otimes \nabla\delta(x)+\lambda \nabla\delta(x)\otimes eu(x',x_n-\lambda\delta(x))e_n]\,\de\lambda\\
&-\int_1^2\lambda\psi(\lambda)\int_1^\lambda (eu)_{nn}(x',x_n-\mu\delta(x))\delta(x)\de \mu\,\nabla^2\delta(x)\de\lambda
\end{align*}
for a.e.\ $x\in Q_r\setminus\Omega$. Therefore, it is natural to define $\tilde f\colon Q_r\to\R^{n\times n}$ as $\tilde f=f$ in $Q_r\cap\Omega$ and
\begin{equation}\label{eq:extension2}
\begin{aligned}
\tilde f(x)=&\int_1^2\psi(\lambda)[f(x',x_n-\lambda\delta(x))+\lambda^2f_{nn}(x',x_n-\lambda\delta(x))\nabla \delta (x)\otimes \nabla \delta (x))]\,\de\lambda \\
&-\int_1^2\lambda\psi(\lambda)[f(x',x_n-\lambda\delta(x))e_n\otimes \nabla\delta(x)+\lambda \nabla\delta(x)\otimes f(x',x_n-\lambda\delta(x))e_n]\,\de\lambda \\
&-\int_1^2\lambda\psi(\lambda)\int_1^\lambda f_{nn}(x',x_n-\mu\delta(x))\delta(x)\de \mu\,\nabla^2\delta(x)\,\de\lambda
\end{aligned}
\end{equation}
for a.e.\ $x\in Q_r\setminus\Omega$, and similarly for $\tilde g$. Therefore, we get that $e\tilde u=\tilde f+\tilde g$ a.e.\ on $Q_r$. Eventually, we define $\tilde p_\lambda\colon Q_r\to [1,\infty)$ as
$$\tilde p_\lambda(x):=p(x',x_n-\lambda\delta(x))\quad\text{for every $x\in Q_r$ and $\lambda\in[1,2]$},$$
 and similarly $\tilde p_\lambda$. By construction, $\tilde p_\lambda=p$ on $Q_r\cap \Omega$ and it satisfies
$$p^-\le \tilde p_\lambda^-\le \tilde p_\lambda^+\le p^+\quad\text{for every $\lambda\in[1,2]$}.$$
Moreover, for every $\lambda\in[1,2]$, by~\eqref{eq:logcon} and~\eqref{eq:der_delta} we have
\begin{equation}\label{eq:locHolder1}
|\tilde p_\lambda(x)-\tilde p_\lambda(y)|\le \frac{c_{\log}(p)}{\log(\e+1/((1+2c_2)|x-y|)}\le \frac{c_3}{\log(\e+1/|x-y|)}\quad\text{for every $x,y\in Q_r$}
\end{equation}
for a constant $c_3=c_3(R,L,n,p)>0$. Hence, by Remark~\ref{rem:H_logglob} we have that $p_\lambda\in\mathcal P^{\log}_b(Q_r)$ for every $\lambda\in[1,2]$ and the $\log$-Hölder constant of $\tilde p_\lambda$ is uniformly bounded with respect to $\lambda\in[1,2]$. We define $\tilde p\colon Q_r\to [1,\infty)$ as
$$\tilde p(x):=\min_{\lambda\in[1,2]}p_\lambda(x)\quad\text{for $x\in Q_r$}.$$
Clearly $\tilde p=p$ on $Q_r\cap\Omega$ and $p^-\le \tilde p^-\le \tilde p_\lambda\le p^+$. We claim that $\tilde p\in \mathcal P^{\log}_b(Q_r)$. Indeed, since the map $ \tilde \lambda\mapsto p_\lambda(x)$ is continuous for every fixed $x\in Q_r$, we can find a countable dense set $\{\lambda_i\}_i\subset[1,2]$ such that 
$${\tilde p(x)=\min_{i\in\N} \tilde p_{\lambda_i}(x) \quad\text{for all $x\in Q_r$}.}$$
Therefore, setting $\tilde p_k:=\min_{i=1,\dots,k} \tilde p_{\lambda_i}$ for every $k\in\N$, by~\eqref{eq:locHolder1} and the fact that the minimum of two locally $\log$-Hölder continuous functions is again locally $\log$-Hölder continuous we have
\begin{equation}\label{eq:locHolder2}
|\tilde p_k(x)-\tilde p_k(y)|\le \frac{c_3}{\log(\e+1/|x-y|)}\quad\text{for every $x,y\in Q_r$},
\end{equation}
and $\tilde p_k(x)\to \tilde p(x)$ as $k\to\infty$ for every $x\in Q_r$. This gives that $\tilde p$ satisfies~\eqref{eq:locHolder2} and proves the claim. Similarly, we can define $\tilde q\in\mathcal P^{\log}_b(Q_r)$. Since $\tilde p_\lambda(x)\le \tilde q_\lambda(x)$ for every $x,y\in Q_r$ and for every $\lambda\in[1,2]$, we conclude that $\tilde p(x)\le \tilde q(x)$ for every $x\in Q_r$. 

It remains to prove that $\tilde f$ and $\tilde g$ satisfy~\eqref{eq:fgest} for a constant $C=C(R,L,n,p,q)>0$. We only discuss the argument for $\tilde f$. Since the extension operator~\eqref{eq:extension2} is linear in $f$, in order to prove~\eqref{eq:fgest} it is enough to find a constant $\overline C=\overline C(R,L,n,p)>0$ satisfying
\begin{equation}\label{eq:estimate2}
\int_{Q_r\setminus\Omega}|\tilde f(x)|^{\tilde p(x)}\,\de x\le \overline C
\end{equation}
for every $f\in L^{p(\cdot)}(Q_R\cap\Omega; \R^{n \times n} )$ with $\|f\|_{L^{\tilde p(\cdot)}(Q_R\cap\Omega)}\le 1$. In view of the bounds~\eqref{eq:delta} and~\eqref{eq:der_delta}, we can find a constant $c_4=c_4(R,L,n) \ge 1 $ such that
$$|\tilde f(x)|\le c_4\int_1^2|f(x',x_n-\lambda\delta(x))|\,\de\lambda\quad\text{for a.e.\ $x\in\ Q_r$}.$$
Therefore, by Jensen's inequality and~\eqref{lastpage} we obtain 
\begin{equation}\label{eq:tildef}
\begin{aligned}
|\tilde f(x)|^{\tilde p(x)}&\le c_4^{p^+} \int_1^2|f(x',x_n-\lambda\delta(x))|^{\tilde p(x)}\,\de\lambda\\
&\le c_4^{p^+} \left(1+\int_1^2|f(x',x_n-\lambda\delta(x))|^{p(x',x_n-\lambda\delta(x))}\,\de\lambda\right)
\end{aligned}
\end{equation}
for a.e.\ $x\in Q_r$, where we used the elementary inequality $t^p\le 1+t^q\quad\text{for every $t\in[0,\infty)$ and $p\le q$}$.

To simplify the notation, we define
$$h(x):=|f(x)|^{p(x)}\quad\text{for a.e.\ $x\in Q_R\cap\Omega$}.$$
Notice that we can write 
$$Q_r\setminus\Omega=\{x=(x',x_n)\in\R^n\,:\, x'\in(-r,r)^{n-1},\,x_n\in(\varphi(x'),r)\}.$$
We fix $x'\in(-r,r)^{n-1}$ such that $\varphi(x')<r$. By~\eqref{eq:delta} for every $x_n\in(\varphi(x'),r)$ we deduce
\begin{align*}
\int_1^2 h(x',x_n-\lambda\delta(x))\,\de\lambda&=\int_{\delta(x)}^{2\delta(x)}\frac{1}{\delta(x)}h(x',x_n-\mu)\,\de \mu\\
&\le \frac{1}{x_n-\varphi(x')}\int_{2(x_n-\varphi(x'))}^{2c_1(x_n-\varphi(x'))}h(x',x_n-\mu)\,\de \mu\\
&=\frac{1}{x_n-\varphi(x')}\int_{x_n-2c_1(x_n-\varphi(x'))}^{x_n-2(x_n-\varphi(x'))}h(x',\nu)\,\de\nu.
\end{align*}
 Recall that $c_2 \ge 1$ and define 
\begin{align*}
&A(x'):=\Big\{(x_n,\nu)\in\R^2\,:\, x_n\in (\varphi(x'),r),\,\nu\in\big(x_n-2c_1(x_n-\varphi(x')),x_n-2(x_n-\varphi(x'))\big)\Big\},\\ 
&B(x'):=\Big\{(x_n,\nu)\in\R^2\,:\, \nu\in\big(2c_1\varphi(x')-r(2c_1-1),\varphi(x')\big),\,x_n\in\big(\frac{-\nu+2c_1\varphi(x')}{2c_1-1},-\nu+2\varphi(x')\big)\Big\}.
\end{align*}
One can easily check that $A(x')\subset B(x')$. By Tonelli's theorem we get 
\begin{align*}
&\int_{\varphi(x')}^r\frac{1}{x_n-\varphi(x')}\left(\int_{x_n-2c_1(x_n-\varphi(x'))}^{x_n-2(x_n-\varphi(x'))}h(x',\nu)\,\de \nu\right)\,\de x_n\\
&\le \int_{2c_1\varphi(x')-r(2c_1-1)}^{\varphi(x')}h(x',\nu)\left(\int_{\frac{-\nu+2c_1\varphi(x')}{2c_1-1}}^{-\nu+2\varphi(x')}\frac{1}{x_n-\varphi(x')}\,\de x_n\right)\,\de \nu\\
& = \int_{2c_1\varphi(x')-r(2c_1-1)}^{\varphi(x')}h(x',\nu) \log(2c_1-1) \,\de \nu \le\log(2c_1-1)\int_{-R}^{\varphi(x')}h(x',\nu)\,\de\nu,
\end{align*}
since $r$ satisfies~\eqref{eq:r}. Hence, by~\eqref{eq:tildef} and by applying again Tonelli's theorem we conclude that
\begin{align*}
\int_{Q_r\setminus\Omega}|\tilde f(x)|^{\tilde p(x)}\,\de x&\le c_4^{p^+} \left(|Q_R|+\int_{ (-r,r)^{n-1}}\int_{\varphi(x')}^r\int_1^2 h(x',x_n-\lambda\delta(x))\,\de\lambda\,\de x_n\,\de x'\right)\\
&\le c_4^{p^+}\left(|Q_R|+\log(2c_1-1)\int_{(-r,r)^{n-1}}\int_{-R}^{\varphi(x')}h(x',\nu)\,\de\nu\,\de x'\right)\\
&\le c_4^{p^+}\left(|Q_R|+\log(2c_1-1)\int_{Q_R\cap\Omega}|f(x)|^{p(x)}\right)\\
&\le c_4^{p^+}\left(|Q_R|+\log(2c_1-1)\right),
\end{align*}
being $\|f\|_{L^{p(\cdot)}(Q_R\cap\Omega)}\le 1$. This gives~\eqref{eq:estimate2} and concludes the proof.
\end{proof}

\subsection{Proofs of the equicoerciveness and the $\Gamma$-convergence in the weak topology }

In this part of the appendix, we give a short proof of the compactness and the $\Gamma$-convergence result of Section~\ref{sec:4}. In order to extend the analogous results of~\cite{AgDMDS} to the case of variable exponents, we first need a version of the rigidity result for the $g$-modular, which is Theorem~\ref{thm:gpxrig}. The proof follows the one of~\cite[Lemma~3.1]{AgDMDS}, and it is much easier than Theorem~\ref{thm:pxrig}, since it relies only on the rigidity for $p=2$ proved in~\cite{FJM}. Therefore, it holds true by assuming only measurability for $p$.

First of all, we recall some properties of the function $g$ introduced in~\eqref{eq:g}. For all $q\in [1,2]$ the function $t\mapsto g(q,t)$ is continuous, increasing, and convex on $[0,\infty)$, while for all $t\in[0,\infty)$ the function $q\mapsto g(q,t)$ is continuous and non decreasing on $[1,2]$. 
% Indeed, let $1\le q_1<q_2\le 2$ and let us consider the function
% $$f(t):=q_1t^{q_2}-q_2t^{q_1}\quad\text{for $t\in[1,\infty)$}.$$
% Since
% $$f'(t)=q_1q_2(t^{q_2-1}-t^{q_1-1})\ge 0\quad\text{for all $t\in[1,\infty)$},$$
% we deduce $f(t)\ge f(1)=q_1-q_2$ for all $t\in[1,\infty)$, which gives
% $$g(q_1,t)\le g(q_2,t)\quad\text{for all $t\in[1,\infty)$}.$$
Moreover, we have
\begin{equation}\label{eq:gest}
g(q,t)\le \min\{t^q,t^2\}\quad\text{for all $q\in[1,2]$ and $t\in[0,\infty)$},
\end{equation}
and 
\begin{equation}\label{eq:gqst}
g(q,s+t)\le 2\left(g(q,s)+g(q,t)\right)\quad\text{for all $q\in[1,2]$ and for all $s,t\in[0,\infty)$}.
\end{equation}
 The latter follows from the convexity of $t\mapsto g(q,t)$ and the inequality
$$g(q,2s)\le 4g(q,s)\quad\text{for all $q\in[1,2]$ and for all $s\in[0,\infty)$}.$$
Furthermore, for all $M>0$ there exists a constant $C=C(M)>0$ such that
\begin{align}
t^2&\le C g(q,t)\quad\text{for all $q\in[1,2]$ and for all $t\in[0,M]$},\label{eq:CM1}\\
t^q&\le C g(q,t)\quad\text{for all $q\in[1,2]$ and for all $t\in[M,\infty)$}.\label{eq:CM2}
\end{align}

% More precisely, we have
% \begin{align*}
% t^2&\le 2\max\{1,M\}g(q,t)\quad\text{for all $q\in[1,2]$ and for all $t\in[0,M]$},\\
% t^q&\le 2\max\{1,M^{-1}\}g(q,t)\quad\text{for all $q\in[1,2]$ and for all $t\in[M,\infty)$},
% \end{align*}
% see additional details file

\begin{theorem}\label{thm:gpxrig}
Let $\Omega\subset\R^n$ be a bounded domain with Lipschitz boundary and let $p\in\mathcal P_b(\Omega)$. Assume that
$$1<p^-\le p^+\le 2.$$
There exists a constant $C=C(\Omega,n)>0$ such that for all $u\in W^{1,p(\cdot)}(\Omega;\R^n)$ there exists a constant rotation $R\in SO(n)$ satisfying
\begin{equation}\label{eq:gpxrig}
\int_\Omega g\big(p(x),|\nabla u(x)-R|\big)\,\de x\le C\int_\Omega g\big(p(x),d(\nabla u(x),SO(n))\big)\,\de x.
\end{equation}
\end{theorem}

\begin{proof}
Let $u\in W^{1,p(\cdot)}(\Omega;\R^n)$. Then $u\in W^{1,1}(\Omega;\R^n)$ and for $\lambda=2\sqrt{n}$ we consider the Lipschitz function $v\colon \Omega\to \R^n$ provided by Lemma~\ref{lem:lusin}. Let $R\in SO(n)$ be the constant rotation associated to $v$ given by the rigidity result for $p=2$, see~\cite[Theorem~3.1]{FJM}. It holds that 
\begin{equation}\label{eq:rigest3}
\int_\Omega |\nabla v(x)-R|^2\,\de x\le C\int_\Omega d(\nabla v(x),SO(n))^2\,\de x
\end{equation}
for a constant $C=C(\Omega,n)>0$. 

By the monotonicity of $t\mapsto g(p(x),t)$,~\eqref{eq:gest}, and~\eqref{eq:gqst}, for a.e.\ $x\in\Omega$ we have
\begin{equation}\label{eq:rigest1}
\begin{aligned}
g(p(x),|\nabla u(x)-R|)&\le g\big(p(x),|\nabla u(x)-\nabla v(x)|+|\nabla v(x)-R|\big)\\
&\le 2\left(|\nabla u(x)-\nabla v(x)|^{p(x)}+|\nabla v(x)-R|^2\right).
\end{aligned}
\end{equation}
Let us consider the first term in the right-hand side of~\eqref{eq:rigest1}. We claim that there exists a constant $C=C(\Omega,n)>0$ such that
\begin{equation*}
\int_\Omega |\nabla u(x)-\nabla v(x)|^{p(x)}\,\de x\le C \int_{\{x\in\Omega\,:\, |\nabla u(x)|>2\sqrt{n}\}} |\nabla u(x)|^{p(x)}\,\de x.
\end{equation*}
Indeed, by (i) and (iii) of Lemma~\ref{lem:lusin} we have
\begin{align*}
&\int_\Omega |\nabla u(x)-\nabla v(x)|^{p(x)}\,\de x\\
&=\int_{\{x\in\Omega\,:\,u(x)\neq v(x)\}}|\nabla u(x)-\nabla v(x)|^{p(x)}\,\de x\\
&\le 2\int_{\{x\in\Omega\,:\,u(x)\neq v(x)\}}(|\nabla u(x)|^{p(x)}+|\nabla v(x)|^{p(x)})\,\de x \\
&\le 2\int_{\{x\in\Omega\,:\, |\nabla u(x)|>2\sqrt{n}\}} |\nabla u(x)|^{p(x)}\,\de x +2\int_{\{x\in\Omega\,:\,u(x)\neq v(x)\}}(2\sqrt{n})^{p(x)}(C^{p(x)}+1)\,\de x\\
&\le 2\int_{\{x\in\Omega\,:\, |\nabla u(x)|>2\sqrt{n}\}} |\nabla u(x)|^{p(x)}\,\de x+2(2\sqrt{n})^2(1+C^2)|\{x\in\Omega\,:\,u(x)\neq v(x)\}|\\
&\le 2\int_{\{x\in\Omega\,:\, |\nabla u(x)|>2\sqrt{n}\}} |\nabla u(x)|^{p(x)}\,\de x+2(2\sqrt{n})^2(1+C^2)\int_{\{x\in\Omega\,:\, |\nabla u(x)|>2\sqrt{n}\}} \left|\frac{\nabla u(x)}{2\sqrt{n}}\right|^{p(x)}\,\de x\\
&\le \left(2+ 8 n (1+C^2)\right)\int_{\{x\in\Omega\,:\, |\nabla u(x)|>2\sqrt{n}\}} |\nabla u(x)|^{p(x)}\,\de x.
\end{align*}
For a.e.\ $x\in\Omega$ let $S(x)\in SO(n)$ be such that
\begin{equation}\label{eq:S}
|\nabla u(x)-S(x)|=d(\nabla u(x),SO(n)).
\end{equation}
Then, for a.e.\ $x\in\{x\in\Omega\,:\, |\nabla u(x)|>2\sqrt{n}\}$ we have $d(\nabla u(x),SO(n))=|\nabla u(x)-S(x)|\ge \sqrt{n}$, which by~\eqref{eq:CM2} gives 
\begin{align*}
|\nabla u(x)|^{p(x)}&\le 2\big(|\nabla u(x)-S(x)|^{p(x)}+\sqrt{n}^{p(x)}\big)\\
&\le 4d(\nabla u(x),SO(n))^{p(x)}\le Cg\big(p(x),d(\nabla u(x),SO(n))\big)
\end{align*}
for a constant $C=C(n)>0$. Therefore, we obtain
\begin{equation}\label{eq:rigest2}
\int_\Omega |\nabla u(x)-\nabla v(x)|^{p(x)}\,\de x\le C\int_\Omega g\big(p(x),d(\nabla u(x),SO(n))\big)\,\de x
\end{equation}
for a constant $C=C(\Omega,n)>0$

It remains to consider the second term in the right-hand side of~\eqref{eq:rigest1}. We claim that
\begin{equation}\label{eq:rigest4}
d(\nabla v(x),SO(n))^2\le C\left(|\nabla u(x)-\nabla v(x)|^{p(x)}+g\big(p(x),d(\nabla u(x),SO(n))\big)\right)
\end{equation}
for a constant $C=C(\Omega,n)>0$. For a.e.\ $x\in\Omega$ we take $S(x)\in SO(n)$ which satisfies~\eqref{eq:S}. We have two cases, namely (i) $|\nabla u(x)-S(x)|\le 1$ and (ii) $|\nabla u(x)-S(x)|> 1$. 
 
(i) Let $x\in\Omega$ be such that $|\nabla u(x)-S(x)|\le 1$. Then
 $$|\nabla u(x)-\nabla v(x)|\le |\nabla u(x)-S(x)|+| S(x) |+|\nabla v(x)|\le 1+\sqrt{n}+2\sqrt{n}C.$$
 Therefore, there exists a constant $C=C(\Omega,n)>0$ such that 
 $$|\nabla u(x)-\nabla v(x)|^2\le C |\nabla u(x)-\nabla v(x)|^{p(x)},$$
 and by~\eqref{eq:CM1} we conclude that 
 \begin{align*}
 d(\nabla v(x),SO(n))^2&\le |\nabla v(x)-S(x)|^2\le 2|\nabla u(x)-\nabla v(x)|^2+2d(\nabla u(x),SO(n))^2\\
 &\le C \left(|\nabla u(x)-\nabla v(x)|^{p(x)}+g(p(x),d(\nabla u(x),SO(n)))\right)
 \end{align*}
 for a constant $C=C(\Omega,n)>0$. 

(ii) Let $x\in\Omega$ such that $|\nabla u(x)-S(x)|> 1$. Since
 $$|\nabla v(x)-S(x)|^2\le 2|\nabla v(x)|^2+2n\le 8C^2n+2n,$$
 by~\eqref{eq:CM2} we get
 \begin{align*}
 d(\nabla v(x),SO(n))^2&\le |\nabla v(x)-S(x)|^2\le (8C^2n+2n)|\nabla v(x)-S(x)|^{p(x)}\\
 &\le C \left(|\nabla u(x)-\nabla v(x)|^{p(x)}+g(p(x),d(\nabla u(x),SO(n)))\right)
 \end{align*}
 for a constant $C=C(\Omega,n)>0$. This shows~\eqref{eq:rigest4}. By combining~\eqref{eq:rigest3},~\eqref{eq:rigest1},~\eqref{eq:rigest2}, and~\eqref{eq:rigest4} we get~\eqref{eq:gpxrig}.
\end{proof}

\begin{remark}\label{rem:norm-mod}
As we already discussed in Remark~\ref{rem:on condition (ii)}, condition (iii) of Lemma~\ref{lem:lusin} has to be used in the above proof. Notice that conditions (i) and (ii) therein are actually not sufficient to this aim, as continuity of the maximal operator in variable Lebesgue spaces can only be formulated in terms of the corresponding norms and not as an integral inequality. 
\end{remark}

As a consequence of the $g$-rigidity result, we obtain the following result.

\begin{lemma}\label{lem:IRest}
Let $\Omega\subset\R^n$ be a bounded domain with Lipschitz boundary and let $p\in\mathcal P_b(\Omega)$ with 
$$1<p^-\le p^+\le 2.$$
Assume that $W$ satisfies \emph{(i)--(iv)}. Let $\varepsilon>0$ and let $u_\varepsilon\in W^{1,p(\cdot)}_h(\Omega;\R^n)$. Let $R_\varepsilon\in SO(n)$ be a constant rotation satisfying~\eqref{eq:gpxrig} with $v_\varepsilon(x):=x+\varepsilon u_\varepsilon(x)$ for a.e.\ $x\in\Omega$. There exists a constant $C=C(\Omega,n)>0$ such that
$$|I-R_\varepsilon|^2\le C\varepsilon^2\left[\mathcal F_\varepsilon(u_\varepsilon)+\left(\int_{\partial_D\Omega}|h|\,\de\mathcal H^{n-1}\right)^2\right].$$
\end{lemma}

\begin{proof}
 By using~\eqref{eq:gpxrig}, property (iv) of $W$, and Jensen's inequality we have
\begin{align*}
g\left(p^-,\frac{1}{|\Omega|}\int_\Omega|\nabla v_\varepsilon(x)-R_\varepsilon|\,\de x\right)&\le \frac{1}{|\Omega|}\int_\Omega g\big(p^-,|\nabla v_\varepsilon(x)-R_\varepsilon|\big)\,\de x\\
&\le \frac{1}{|\Omega|}\int_\Omega g\big(p(x),|\nabla v_\varepsilon(x)-R_\varepsilon|\big)\,\de x\le C \varepsilon^2\mathcal F_\varepsilon(u_\varepsilon).
\end{align*}
We now use Lemma~\ref{lem:poincare2} to deduce that
\begin{align*}
|I-R_\varepsilon|^2\le C\left[\left(\frac{1}{|\Omega|}\int_\Omega|\nabla v_\varepsilon(x)-R_\varepsilon|\,\de x\right)^2+\varepsilon^2\left(\int_{\partial_D\Omega}|h|\,\de\mathcal H^{n-1}\right)^2\right].
\end{align*}
If $ \frac{1}{|\Omega|} \int_\Omega|\nabla v_\varepsilon(x)-R_\varepsilon|\,\de x\le 1$ we conclude by the definition of $g$, while if $ \frac{1}{|\Omega|} \int_\Omega|\nabla v_\varepsilon(x)-R_\varepsilon|\,\de x>1$, then $C\varepsilon^2\mathcal F_\varepsilon(u_\varepsilon)>\frac{1}{2}$ and we conclude by observing that $|I-R_\varepsilon|^2\le 4n$.
\end{proof}

Based on this, we can give the proof of Proposition~\ref{lem:equicoerv}. 

\begin{proof}[Proof of Proposition~\ref{lem:equicoerv}]
For all $\varepsilon\in(0,1)$ let $R_\varepsilon\in SO(n)$ be the constant matrix given by~\eqref{eq:gpxrig} for the function $v_\varepsilon(x):=x+ \varepsilon u_\varepsilon(x)$ for a.e.\ $x\in\Omega$. By using~\eqref{eq:gest},~\eqref{eq:gqst}, and the monotonicity of $g$ we have
\begin{align*}
\int_\Omega g\big(p(x),|\varepsilon\nabla u_\varepsilon(x)|\big)\,\de x&\le \int_\Omega g\big(p(x),|\nabla v_\varepsilon(x)-R_\varepsilon|+|I-R_\varepsilon|\big)\,\de x\\
&\le 2\int_\Omega g\big(p(x),|\nabla v_\varepsilon(x)-R_\varepsilon|\big)\,\de x+2|\Omega||I-R_\varepsilon|^2\\
&\le C\left(\int_\Omega g\big(p(x),d(\nabla v_\varepsilon(x),SO(n))\big)\,\de x+|I-R_\varepsilon|^2\right),
\end{align*}
for a constant $C=C(\Omega,n)>0$. By Lemma~\ref{lem:IRest} and assumption (iv) of $W$, we can find another constant $C=C(\Omega,n)>0$ such that
\begin{align}\label{eq: LLLLLLLL}
\int_\Omega g(p(x),|\varepsilon\nabla u_\varepsilon(x)|)\,\de x&\le C\varepsilon^2\left[\mathcal F_\varepsilon(u_\varepsilon)+\left(\int_{\partial_D\Omega}|h|\,\de\mathcal H^{n-1}\right)^2
\right].
\end{align}
In particular, this implies 
$$\int_{\{x\in\Omega\,:\,|\varepsilon\nabla u_\varepsilon(x)|\le 1\}}|\varepsilon\nabla u_\varepsilon(x)|^2\,\de x\le 2C\varepsilon^2\left[\mathcal F_\varepsilon(u_\varepsilon)+\left(\int_{\partial_D\Omega}|h|\,\de\mathcal H^{n-1}\right)^2
\right].$$
Since $t^q\le t^2+1$ for all $q\in[1,2]$ and $t\in[0,\infty)$, we deduce that there exists a constant $C=C(\Omega,n)>0$ such that
\begin{equation}\label{eq:coerc1}
\int_{\{x\in\Omega\,:\,|\varepsilon\nabla u_\varepsilon(x)|\le 1\}}|\nabla u_\varepsilon(x)|^{p(x)}\,\de x\le C\left[1+\mathcal F_\varepsilon(u_\varepsilon)+\left(\int_{\partial_D\Omega}|h|\,\de\mathcal H^{n-1}\right)^2\right]
\end{equation}
for all $\varepsilon\in(0,1)$. 

On the other hand, by~\eqref{eq: LLLLLLLL} we have
\begin{align*}
\varepsilon^{p^+}\int_{\{x\in\Omega\,:\,|\varepsilon\nabla u_\varepsilon(x)|> 1\}}|\nabla u_\varepsilon(x)|^{p(x)}\,\de x&\le \int_{\{x\in\Omega\,:\,|\varepsilon\nabla u_\varepsilon(x)|> 1\}}|\varepsilon\nabla u_\varepsilon(x)|^{p(x)}\,\de x\\
&\le 2C\varepsilon^2\left[\mathcal F_\varepsilon(u_\varepsilon)+\left(\int_{\partial_D\Omega}|h|\,\de\mathcal H^{n-1}\right)^2\right],
\end{align*}
which gives 
\begin{equation}\label{eq:coerc2}
\int_{\{x\in\Omega\,:\,|\varepsilon\nabla u_\varepsilon(x)|> 1\}}|\nabla u_\varepsilon(x)|^{p(x)}\,\de x\le C\left[ \mathcal F_\varepsilon(u_\varepsilon) +\left(\int_{\partial_D\Omega}|h|\,\de\mathcal H^{n-1}\right)^2\right]
\end{equation}
for a constant $C=C(\Omega,n)>0$ and for all $\varepsilon\in(0,1)$. By combining~\eqref{eq:coerc1} and~\eqref{eq:coerc2} we deduce~\eqref{eq:equicoer}.
\end{proof}

Let us finally come to the proof of Theorem~\ref{thm:gamma_con}. Since $g(p(x),t)\ge g(p^-,t)$ for a.e.\ $x\in \Omega$ and every $t\in[0,\infty)$ and $1<p^-\le 2$, we can apply~\cite[Theorem~2.4]{AgDMDS} to deduce that ${\mathcal F}_{\varepsilon_j}\xrightarrow{\Gamma}\mathcal{F}\quad\text{as $j\to \infty$}$ in the weak topology of $W^{1,p^-}_h(\Omega;\R^n)$. More precisely, in~\cite{AgDMDS} it has been shown that for each sequence $\lbrace u_{j}\rbrace_j \subset W^{1,p^-}_h(\Omega;\R^n)$ with $u_j \rightharpoonup u$ weakly in $ W^{1,p^-}(\Omega;\R^n)$ we have
\begin{align}\label{gamma1}
\liminf_{j \to \infty} {\mathcal F}_{\varepsilon_j}(u_j) \ge \mathcal{F}(u), 
\end{align}
and for every $u \in W^{1,2}_h(\Omega;\R^n)$ we find a sequence $\lbrace u_j\rbrace_j\subset W^{1,2}_h(\Omega;\R^n)$ converging strongly in $W^{1,2}(\Omega;\R^n)$ such that 
\begin{align}\label{gamma2}
\lim_{j \to \infty} {\mathcal F}_{\varepsilon_j}(u_j) = \mathcal{F}(u). 
\end{align}

\begin{proof}[Proof of Theorem~\ref{thm:gamma_con}]
The $\Gamma$-liminf inequality in the weak topology of $W^{1,p(\cdot)}_h(\Omega;\R^n)$ follows from \eqref{gamma1} and the weak continuity of the embedding $W^{1,p(\cdot)}(\Omega;\R^n)\subseteq W^{1,p^-}(\Omega;\R^n)$. The $\Gamma$-limsup inequality follows from~\eqref{gamma2} and the fact that the strong topology of $W^{1,2}(\Omega;\R^n)$ is stronger than the weak topology of $W^{1,p(\cdot)}_h(\Omega;\R^n)$. 
\end{proof} 

%\section*{Statements and declarations}
%
%\noindent  {\bf Conflict of interest.} The authors have no conflicts of interest to declare.
%
%\smallskip
%
%\noindent {\bf Data availability.} The authors do not analyse or generate any datasets.
%
%\smallskip
%
%\noindent  {\bf Acknowledgments.} This work was supported by the Austrian Science Fund through the projects ESP-61 and P35359-N, by the DFG projects FR 4083/3-1, FR 4083/5-1, by the Deutsche Forschungsgemeinschaft (DFG, German Research Foundation) under Germany's Excellence Strategy EXC 2044 -390685587, Mathematics M\"unster: Dynamics--Geometry--Structure, by the Italian Ministry of Education and Research through the PRIN 2017 project No. 2017BTM7SN, and by project Starplus 2020 Unina Linea 1 "New challenges in the variational modeling of continuum mechanics" from the University of Naples Federico II and Compagnia di San Paolo. Finally, M.C.\ and F.S.\ are partially supported by Gruppo Nazionale per l’Analisi Matematica, la Probabilità e le loro Applicazioni (GNAMPA-INdAM). 

\section*{Acknowledgments}
This work was supported by the Austrian Science Fund through the projects ESP-61 and P35359-N, by the DFG projects FR 4083/3-1, FR 4083/5-1, by the Deutsche Forschungsgemeinschaft (DFG, German Research Foundation) under Germany's Excellence Strategy EXC 2044 -390685587, Mathematics M\"unster: Dynamics--Geometry--Structure, by the Italian Ministry of Education and Research through the PRIN 2017 project No. 2017BTM7SN, and by project Starplus 2020 Unina Linea 1 "New challenges in the variational modeling of continuum mechanics" from the University of Naples Federico II and Compagnia di San Paolo. Finally, M.C.\ and F.S.\ are partially supported by Gruppo Nazionale per l’Analisi Matematica, la Probabilitá e le loro Applicazioni (GNAMPA-INdAM).

\end{document}